\documentclass[a4paper,fleqn]{cas-sc}

\usepackage[square,numbers,sort&compress]{natbib}
\usepackage{graphicx}          % Include this line if your
\usepackage{lineno,hyperref}
\usepackage{amsfonts}
\usepackage{amsmath}
\usepackage{amssymb}
\usepackage{amscd,multirow}
\usepackage{graphicx}
\usepackage{epstopdf}
\usepackage[titletoc]{appendix}
\usepackage{subfigure}
\usepackage[justification=centering]{caption}
\usepackage{subeqnarray,cases}
\usepackage[ruled]{algorithm2e}
\numberwithin{equation}{section}

\allowdisplaybreaks[4]
\newtheorem{theorem}{Theorem}
\newtheorem{corollary}{Corollary}
\newtheorem{lemma}{Lemma}
\newtheorem{proposition}{Proposition}
\newtheorem{remark}{Remark}
\newtheorem{example}{Example}
\newtheorem{assumption}{Assumption}

\newenvironment{proof}{\noindent {\textbf{Proof.}}}

\begin{document}
	\let\WriteBookmarks\relax
	\def\floatpagepagefraction{1}
	\def\textpagefraction{.001}
	\let\printorcid\relax
	
	\shorttitle{Convergence analysis of accelerated algorithms via a mixed-order dynamical system...}
	\shortauthors{G.H. Li, H.Y. Zhao, X.K. Sun}
	%\begin{frontmatter}
	
	\title [mode = title]{Convergence analysis of accelerated algorithms via a mixed-order dynamical system for separable nonsmooth convex optimization}\tnotemark[1]
	
	\tnotetext[1]{This research was supported by the Natural Science Foundation of Chongqing, China (CSTB2025NSCQ-GPX0137), the Science and Technology Research Program of Chongqing Municipal Education Commission (KJQN202500840), the National Natural Science Foundation of China (12001070), and the Team Building Project for Graduate Tutors in Chongqing (yds223010).}
	%\tnotetext[2]{The second title footnote which is a longer text matter
		%   to fill through the whole text width and overflow into
		%   another line in the footnotes area of the first page.}

	\author[1]{Geng-Hua Li}\ead{ligh2008cqu@163.com}
	\cormark[1]
	\cortext[1]{Corresponding author.}
	\author[1]{Hai-Yi Zhao}\ead{zhaohaiyii@163.com}
	\author[1]{Xiangkai Sun}\ead{sunxk@ctbu.edu.cn} 
	
	\address[1]{Chongqing Key Laboratory of  Statistical Intelligent Computing and Monitoring, School of Mathematics and Statistics, Chongqing Technology and Business University, Chongqing, 400067, China}
	
	%\fntext[fn1]{This is the first author footnote, but is common to third
		%  author as well.}
	%\fntext[fn2]{Another author footnote, this is a very long footnote and
		%  it should be a really long footnote. But this footnote is not yet
		%  sufficiently long enough to make two lines of footnote text.}
	
	\nonumnote{}
	
	\begin{abstract}
		For a linear equality constrained convex optimization problem involving two objective functions with a ``nonsmooth" + ``nonsmooth" composite structure, we study two algorithms derived from a mixed-order dynamical system which incorporates time scales and a Tikhonov regularization term. We observe that different types of multipliers lead to distinct algorithms. For the implicit multiplier and semi-implicit multiplier, we develop a new primal-dual joint algorithm and a new splitting algorithm, respectively. Our proposed joint algorithm can reduce to an algorithm for solving the corresponding non-separable linearly constrained convex optimization problem. Then, we establish the nonergodic convergence properties of all our proposed algorithms. Moreover, we derive that the sequences generated by these algorithms strongly converge to the minimal norm solution. Finally, numerical experiments are conducted to validate the practical performance of the proposed algorithms.
	\end{abstract}
	
	%\begin{graphicalabstract}
	%\includegraphics{figs/cas-grabs.pdf}
	%\end{graphicalabstract}
	
	\begin{highlights}
		\item Design algorithms via a mixed-order dynamics for nonsmooth separable convex problems.
		\item Demonstrate the nonergodic rates for convex and partially strongly convex functions.
		\item Demonstrate the strong convergence of the sequences generated by our algorithms.
	\end{highlights}


	\maketitle
	
	\section{Introduction}
	
	Let $\mathcal{X,Y,Z}$ be three real Hilbert spaces with inner product $\langle \cdot,\cdot\rangle$ and norm $\|\cdot\|$. Consider the following separable convex optimization problem:
	\begin{eqnarray}\label{P1}
		\begin{array}{ll}
			&\mathop{\mbox{min}}\limits_{x\in \mathcal{X}, y\in \mathcal{Y}}~\varPhi  (x,y):=f(x) + g(y)\\
			&~~~\mbox{s.t.}~~~~~Ax+By=b,
		\end{array}
	\end{eqnarray}
	where $f:\mathcal{X} \rightarrow \mathbb{R}\cup\{+\infty\}$ and $g:\mathcal{Y} \rightarrow \mathbb{R}\cup\{+\infty\}$ are proper, convex and lower semi-continuous functions; $A: \mathcal{X} \rightarrow \mathcal{Z}$ and $B: \mathcal{Y} \rightarrow \mathcal{Z}$ are continuous linear operators and $b\in \mathcal{Z}$; The optimal solution set $\mathbb{S}$ of problem $\eqref{P1}$ is nonempty.

	Owing to its versatility, problem \eqref{P1} plays a crucial role in applications ranging from image processing, machine learning, distributed optimization and signal recovery (see, e.g.,\cite{1,27,11,41}). When $y=0$, problem \eqref{P1} collapses to the following non-separable linear equality constrained optimization problem:
	\begin{equation}\label{P2}
		\underset{x\in \mathcal{X}}{\min}\,\,f\left( x \right),\,\, \mathrm{s}.\mathrm{t}. \,\, Ax=b.
	\end{equation}
	The Lagrangian function $\mathcal{L}:\mathcal{X}\times\mathcal{Y}\times\mathcal{Z}\to\mathbb{R}\cup\{+\infty\}$, associated with the problem \eqref{P1}, is defined by
	\begin{equation*}
		\mathcal{L}(x,y,\lambda) = f(x)+g(y)+\langle \lambda,Ax+By-b\rangle,
	\end{equation*}
	where $ \lambda $ is the Lagrangian multiplier. Accordingly, the augmented Lagrangian function corresponding to problem \eqref{P1} is defined as
	\begin{equation*}\label{zgLa}
		\mathcal{L}_\theta(x,y,\lambda) = \mathcal{L}(x,y,\lambda)+\frac{\theta}{2}\|Ax+By-b\|^2.
	\end{equation*}
	For problem \eqref{P2}, it reduces to
	\begin{equation*}\label{zgLa2}
		\mathcal{L}_\theta(x,\lambda) = \mathcal{L}(x,\lambda)+\frac{\theta}{2}\|Ax-b\|^2.
	\end{equation*}

	Consider the following saddle point problem associated problem \eqref{P1}
	\begin{equation}\label{max min}
		\underset{\lambda\in \mathcal{Z}}\max\,\underset{x\in \mathcal{X},y\in \mathcal{Y}}\min \,\mathcal{L}(x,y,\lambda).
	\end{equation}
	Let $\Omega$ denote the set of saddle points of Lagrangian function $\mathcal{L}$. This means that $(x^*,y^*,\lambda^*)\in\Omega$ if and only if
	\[ \mathcal{L}(x^*,y^*,\lambda)\leq  \mathcal{L}(x^*,y^*,\lambda^*)\leq  \mathcal{L}(x,y,\lambda^*), \qquad \forall (x,y,\lambda)\in \mathcal{X}\times\mathcal{Y}\times\mathcal{Z}.\]
	Then the optimality conditions for problem \eqref{P1} reads
	\begin{equation*}
		\begin{cases}
			-A^T\lambda^* \in \partial f(x^*),\\
			-B^T\lambda^* \in \partial g(y^*),\\
			Ax^*+By^* -b =0,
		\end{cases}
	\end{equation*}
	where $\partial f(x) = \{\nu\in\mathcal{X}|f(y)-f(x)\geq \langle \nu,y-x\rangle \}$.

	For given $ \mathcal{\lambda} $, we set
	\begin{equation*}
		\mathbb D(\lambda) := \underset{x\in \mathcal{X}, y\in \mathcal{Y}}\min\,\mathcal{L}(x,y,\lambda),\qquad
		\mathcal D(\lambda) := \underset{x\in \mathcal{X}, y\in \mathcal{Y}}{\arg\min}\,\mathcal{L}(x,y,\lambda).
	\end{equation*}
	If the optimal solution of problem \eqref{max min} is achieved by ${\lambda^*},$ then
	\begin{equation*}
		\mathcal D({\lambda}^*) = \{{(x,y)\in \mathcal{X\times Y}}|(x,y)\in \underset{x\in \mathcal{X}, y\in \mathcal{Y}}{\arg\min}\,\mathcal{L}(x,y,\lambda^*)\}\subseteq\mathbb{S}.
	\end{equation*}

	\subsection{Literature review}
	
	In these years, dynamical system schemes have attracted significant attention in optimization research. Dynamical systems provide deep insights into certain existing numerical techniques and develop novel algorithms via temporal discretization. Specifically, the unconstrained optimization problem has been addressed using both continuous dynamical methods \cite{32,4,33,3} and algorithms based on time discretization of continuous-time dynamical systems \cite{2,31,26,7}. Following research on the fast convergence properties of unconstrained convex optimization problems, researchers are now focusing on linearly constrained convex optimization.

	For non-separable linearly constrained optimization problem (\ref{P2}), it is well known that continuous dynamical systems provide an effective approach. Zeng et al. \cite{40} constructed a second-order primal-dual dynamic with asymptotic vanishing viscous damping and proved the convergence rate $ \mathcal{O}(\frac{1}{t^{\min\{2,\frac{2}{3}\alpha\}}}) $ for the primal-dual gap and $ \mathcal{O}(\frac{1}{t^{\min\{1,\frac{1}{3}\alpha\}}}) $ for the feasibility measure with $\alpha>0$ and $\beta=\min\{\frac{1}{2},\frac{3}{2}\alpha\}$. Subsequently, Bo\c{t} et al. \cite{9} extended the dynamical system in \cite{40} by incorporating a time scaling. They not only derived the convergence rates, but also established the weak convergence of the trajectories. To reduce the computational cost, He et al. \cite{20} designed a mixed-order dynamical system with a second-order ODE for the primal variable and a first-order ODE for the dual variable. They showed that this dynamical system achieved a convergence rate of $\mathcal{O}(\frac{1}{\beta(t)})$ for the primal-dual gap.

	On the other hand, some scholars have developed several rapidly convergent algorithms for the non-separable problem (\ref{P2}). One classical algorithm that has been widely studied is the augmented Lagrangian method (ALM) \cite{34} proposed by Rockafellar. So far, numerous variants of ALM have been developed, including those based on Nesterov's extrapolation technique \cite{16,36,22}, the quadratic penalty method \cite{37}, the accelerated linearization method \cite{39}, and other related techniques. In addition to modifying the ALM framework directly, another approach is to discretize continuous dynamical systems to obtain variants of ALM or other novel algorithms for problem \eqref{P2}. Bo\c{t} et al. \cite{10} proposed a new fast augmented Lagrangian algorithm from the second-order dynamical system with vanishing damping. Then they showed that the convergence rates of the primal-dual gap, the feasibility measure and the objective function value are $\mathcal{O}(\frac{1}{k^2})$. Afterwards, He et al. \cite{21} proposed an accelerated primal-dual algorithm by discretizing a mixed-order dynamical system and proved the $\mathcal{O}(\frac{1}{k^2\beta_k})$ convergence rate for the objective residual. Furthermore, Ding et al. \cite{13} proposed an algorithm by discretizing a dynamical systems with viscous damping and determined the iterates generated by the algorithm weakly converges to an optimal solution. Tikhonov regularization can guarantee strong convergence to the minimal norm solution. Following this idea, Zhu et al. \cite{42} derived a new primal-dual algorithm by incorporating Tikhonov regularization into an inertial dynamical system. Their algorithm not only retained the strong convergence property to the minimal norm solution, but also achieved an $O(\frac{1}{k^2})$ convergence rate for the primal-dual gap, objective residual, and feasibility violation. For further details on other numerical algorithms by time discretization, see \cite{24,25,38,28}.

	To better adapt to practical scenarios, approaches for the non-separable linear equality constrained optimization problem (\ref{P2}) have been extended to separable optimization problem (\ref{P1}). This extension also applies to using continuous dynamical systems to solve problem (\ref{P1}). He et al. \cite{19} considered an inertial primal-dual dynamical system featuring two second-order primal variables and one second-order dual variable. Specifically, they investigated the convergence rates under different choices of the damping coefficients. Then, Attouch et al. \cite{6} introduced a time scaling into the aforementioned dynamical system and provided fast convergence properties of the values and the feasibility gap. Moreover, Sun et al. \cite{35} investigated the following Tikhonov regularized mixed-order dynamical system:
	\begin{eqnarray}\label{zsjth}
		\left\{ \begin{array}{ll}
			\ddot{x}(t)+\gamma \dot{x}(t)+\beta(t)\left( \nabla f(x(t))+A^\top \lambda(t)+A^\top( Ax(t)+B y(t)-b)+\epsilon(t)x(t)\right)=0,\\
			\ddot{y}(t)+\gamma \dot{y}(t)+\beta(t)\left( \nabla g(y(t))+B^\top\lambda(t)+B^\top( Ax(t)+B y(t)-b)+\epsilon(t)y(t)\right)=0,\\
			\dot{\lambda }(t)-\beta(t)\left( A( x(t)+\delta \dot{x}(t))+B(y(t)+\delta\dot{y}(t))-b\right)=0.
		\end{array}
		\right.
	\end{eqnarray}
	and they established both the convergence rate and the strong convergence properties of this continuous-time system (\ref{zsjth}). For further continuous dynamical system methods on separable optimization problems, we refer to \cite{8,14}.

	Meanwhile, separable problem (\ref{P1}) has been addressed by a variety of algorithms. Among these, the Alternating Direction Method of Multipliers (ADMM) \cite{admm} is a widely used and influential approach. Over these years, numerous variants have been proposed to further enhance its performance, including symmetrization \cite{j}, proximal preconditioning \cite{17}, Nesterov's extrapolation \cite{23} and parallelization \cite{15}. By discretizing continuous dynamical systems, a broad class of algorithms is derived for problem (\ref{P1}), covering ADMM-type methods as well as other variants, thus establishing a link between continuous-time dynamics and discrete iterative schemes. In the ``Conclusion and perspective" of \cite{19}, He et al. proposed an ADMM-type algorithm by discretizing their proposed dynamical system. This algorithm can be viewed as a variant of the inertial proximal ADMM \cite{5}, though no convergence proof was provided for it. Notably, Chao et al. \cite{admmx} demonstrated that the classic ADMM can be derived from the explicit time discretization of a continuous-time dynamical system they proposed. Recently, Luo et al. \cite{30} presented a unified ordinary differential equation solver framework:
	\begin{equation}\label{luoh}
		\left\{
		\begin{aligned}
			0&{}\in\gamma x'' + (\gamma+\mu_f)x' +\partial_x \mathcal L(x,y,\lambda),\\
			0&{}=\theta\lambda ' -\nabla_\lambda \mathcal L(x+x',y+y',\lambda),\\
			0&{}\in	\beta y'' + (\beta+\mu_g)y' +\partial_y \mathcal L(x,y,\lambda),
		\end{aligned}
		\right.
	\end{equation}
	where $\mu_f,\mu_g$ is the strong convexity parameter and the parameters $(\theta,\gamma,\beta)$ are governed by $\theta'=-\theta$, $\gamma' = \mu_f-\gamma$, $\beta'=\mu_g-\beta$. With various choices of the multipliers, they obtained several algorithms by discretizing this dynamic. For convex objectives, their algorithms achieve a nonergodic convergence rate of $\mathcal{O}(\frac{1}{k})$, while for partially strongly convex ones, they achieve a rate of $\mathcal{O}(\frac{1}{k^2})$.

	To the best of our knowledge, only a few studies exist on splitting algorithms based on the discretization of dynamical systems for separable convex optimization. Motivated by the works \cite{30} and \cite{35}, this paper derives novel algorithms by considering different multipliers in the discretization of the following mixed-order dynamical system:
	\begin{equation}\label{dyn}
		\left\{ \begin{array}{ll}
			\ddot{x}(t)+\gamma_1(t) \dot{x}(t)+\beta(t)\left( \partial f(x(t))+A^\top \lambda(t)+\epsilon(t)x(t)\right)\ni0,\\
			\ddot{y}(t)+\gamma_2(t) \dot{y}(t)+\beta(t)\left( \partial g(y(t))+B^\top\lambda(t)+\epsilon(t)y(t)\right)\ni0,\\
			\dot{\lambda }(t)-\beta(t)\left( A( x(t)+\delta \dot{x}(t))+B(y(t)+\delta\dot{y}(t))-b\right)=0,
		\end{array}
		\right.
	\end{equation}
	where $t\geq t_0>0$, $\gamma_1(t)=\gamma+\mu_f\delta\beta(t)$, $\gamma_2(t)=\gamma+\mu_g\delta\beta(t)$, $\gamma $ is a constant damping coefficient, $\delta $ is a constant extrapolation coefficient, $ \beta:[t_0,+\infty)\rightarrow(0,+\infty)$ is a positive time scale, $\epsilon:  [ t_0,+\infty )\rightarrow [0,+\infty)$ is the Tikhonov regularization function and ${\mu}_f$, ${\mu}_g\geq 0$ correspond to the strong convexity parameters of $f$ and $g$. Specially, ${\mu}_f= 0$ if and only if $f$ is convex, and ${\mu}_g = 0$ if and only if $g$ is convex.
	
	\subsection{Main contributions}
	
	We summarize our main contributions as follows:
	\begin{itemize}
		\item[{\rm (i)}] For the separable nonsmooth convex optimization problem (\ref{P1}), we propose new algorithms, which can be viewed as discrete versions of a novel dynamical system (\ref{dyn}). When the objective functions are convex, we demonstrate that the objective residual and the feasibility violation converge at a rate of $\mathcal{O}(1/\beta_k)$, which is faster than the corresponding continuous-time rate of $\mathcal{O}(1/\sqrt{\beta(t)})$ reported in \cite{35} under certain conditions.
		\item[{\rm (ii)}] Compared with the convergence results in \cite{30}, the sequences generated by our proposed algorithms converge strongly to the minimal norm solution. Additionally, our work replaces the intricate equality constraint between step size and time scales in \cite{30} with a more flexible inequality, thus broadening the range of parameter selection. Furthermore, adopting a different semi-implicit multiplier choice from that in \cite{30}, we obtain Algorithm \ref{al_3} which guarantees a convergence rate of $O(1/k)$ for convex objective functions and $O(1/k^2)$ for partially strongly convex ones.
		\item[{\rm (iii)}] To solve the non-separable linear equality constrained optimization problem (\ref{P2}), we also obtain Algorithm \ref{al_4} which is a simplified version of Algorithm \ref{al_1}. Compared with \cite{13}, Algorithm \ref{al_4} introduces a regularization term so that the sequence converges strongly to the minimal norm solution of problem (\ref{P2}). When $\beta_{k}>k^2$, Algorithm \ref{al_4} exhibits a faster convergence rate than the $O(1/k^2)$ rate established in \cite{42}. If $\epsilon_k$ equals zero, Algorithm \ref{al_4} can achieve the $\mathcal O((1+\frac{\alpha_{min}}{\delta})^{-k})$ convergence rate.
	\end{itemize}
	
	\subsection{Outlines}
	
	This paper is organized as follows. In Section 2, we introduce two algorithms based on the discretization of the mixed-order dynamical system (\ref{dyn}) to solve separable nonsmooth convex optimization problem \eqref{P1}. In Section 3, we present the convergence properties of proposed algorithms. Section 4 demonstrates that under suitable conditions, the iterates $ \{(x_{k},y_{k})\}_{k\geq1} $ generated by our algorithms converge strongly to the minimal norm solution of problem \eqref{P1}. In Section 5, we derive a simplified algorithm from Algorithm \ref{al_1} for the non-separable convex optimization problem and establish its convergence properties. In Section 6, the numerical experiments are presented.
	
	\section{Accelerated algorithms via dynamical system}
	
	In this section, we adopt implicit discretization approach for the mixed-order dynamics (\ref{dyn}). The time discretization technique follows the framework introduced in \cite{10} and \cite{13}. First we define
	\begin{equation*}
		\begin{cases}
			{Z}(t):=\dot{x}(t)+\gamma{x}(t),\\
			{H}(t):=\dot{y}(t)+\gamma{y}(t),
		\end{cases}
	\end{equation*}
	and
	\begin{equation*}
		\begin{cases}
			{Z}^\delta(t):={x}(t)+\delta\dot{x}(t)=\delta{Z}(t)+(1-\delta\gamma){x}(t),\\
			{H}^\delta(t):={y}(t)+\delta\dot{y}(t)=\delta{H}(t)+(1-\delta\gamma){y}(t).
		\end{cases}
	\end{equation*}
	Then for every $t\geq t_0$, (\ref{dyn}) can be rewritten as
	\begin{equation}\label{dyn3}
		\left\{ \begin{array}{ll}
			\dot{Z}(t)\in-\beta(t)\left( \partial f(x(t))+A^\top \lambda(t)+\epsilon(t)x(t)+\mu_f\delta\dot{x}(t)\right),\\
			\dot{H}(t)\in-\beta(t)\left( \partial g(y(t))+B^\top \lambda(t)+\epsilon(t)y(t)+\mu_g\delta\dot{y}(t)\right),\\
			\dot{\lambda}(t)=\beta(t)\left( A{Z}^\delta(t)+B{H}^\delta(t)-b\right),\\
			{Z}(t)=\dot{x}(t)+\gamma{x}(t),\\
			{Z}^\delta(t)={x}(t)+\delta\dot{x}(t),\\
			{H}(t)=\dot{y}(t)+\gamma{y}(t),\\
			{H}^\delta(t)={y}(t)+\delta\dot{y}(t).
		\end{array}
		\right.
	\end{equation}

	Inspired by \cite{30}, we discretize the dynamics (\ref{dyn3}) with a step size $\alpha_k>0$ and vary the choice of multipliers $\bar{\lambda}_{k+1}$, $\hat{\lambda}_{k+1}$ and ${\lambda}_{k+1}$, resulting in the following scheme:
	\begin{equation}\label{disc}
		\left\{ \begin{array}{ll}
			\frac{{Z}_{k+1}-{Z}_k}{\alpha_k}\in-{\beta}_k( \partial f(x_{k+1})+A^\top \bar{\lambda}_{k+1}+\epsilon_k{x}_{k+1}-\mu_f(x_{k+1}-{Z}^\delta_{k+1})),\\
			\frac{{H}_{k+1}-{H}_k}{\alpha_k}\in-{\beta}_k( \partial g(y_{k+1})+B^\top \hat{\lambda}_{k+1}+\epsilon_k{y}_{k+1}-\mu_g(y_{k+1}-{H}^\delta_{k+1})),\\
			\frac{{\lambda}_{k+1}-{\lambda}_k}{\alpha_k}={\beta}_k\left( A{Z}^\delta_{k+1}+B{H}^\delta_{k+1}-b\right),\\
			{Z}_{k+1}=\frac{x_{k+1}-x_k}{\alpha_k}+\gamma{x}_{k+1},\\
			{Z}^\delta_{k+1}=\delta{Z}_{k+1}+(1-\delta\gamma)x_{k+1},\\
			{H}_{k+1}=\frac{y_{k+1}-y_k}{\alpha_k}+\gamma{y}_{k+1},\\
			{H}^\delta_{k+1}=\delta{H}_{k+1}+(1-\delta\gamma)y_{k+1}.
		\end{array}
		\right.
	\end{equation}

	It is worth mentioning that the way to define these multipliers is crucial. We first present a joint algorithm based on an implicit multiplier scheme:
	\begin{equation}\label{label1}
		\bar{\lambda}_{k+1} = \hat{\lambda}_{k+1} = {\lambda}_{k+1}=\lambda_k + \alpha_k\beta_{k}(A{Z}^\delta_{k+1}+B{H}^\delta_{k+1}-b).
	\end{equation}
	Now, we rewrite scheme (\ref{disc}) as Algorithm \ref{al_1}. One feature of Algorithm \ref{al_1} is that ${x}_{k+1}$ and ${y}_{k+1}$ are coupled.
	\begin{algorithm}\small
		\caption{Primal-dual Joint Algorithm}
		\label{al_1}
		{\bf Initialize:} Let $x_1=x_0=\frac{1}{\gamma} Z_1, y_1=y_0=\frac{1}{\gamma} H_1, {\lambda}_1={\lambda}_0$, $\delta>0,\gamma>0$. \\
		\For{$k = 1, 2,\cdots$}{
			{\bf Step1:} Let $\theta_k=(\alpha_k+{\delta})\beta_k$ and $\eta_{f,k}=\gamma+\frac{1}{\alpha_k}+\mu_f\delta\beta_k$ and
			$\eta_{g,k}=\gamma+\frac{1}{\alpha_k}+\mu_g\delta\beta_k,$\\
			$\widetilde{\lambda}_k = \lambda_k -\delta\beta_k(Ax_k+By_k-b),$\\
			$\widetilde{x}_k=x_k+\frac{Z_k-\gamma x_k}{\eta_{f,k}}, \widetilde{y}_k=y_k+\frac{H_k-\gamma y_k}{\eta_{g,k}},$\\
			$(x_{k+1}, y_{k+1}) =\mathop{\arg\min}_{(x,y)\in\mathcal{X}\times\mathcal{Y}}( \mathcal{L}_{\theta_k}(x,y,\widetilde{\lambda}_k)
			+\frac{\eta_{f,k}}{2\alpha_k\beta_k}\left\|x-\widetilde{x}_k\right\|^2+\frac{\eta_{g,k}}{2\alpha_k\beta_k}\left\|y-\widetilde{y}_k\right\|^2 +\frac{{\epsilon}_k}{2}\left\|x\right\|^2+\frac{{\epsilon}_k}{2}\left\|y\right\|^2) ,$\\
			{\bf Step2:}$({Z}_{k+1},{H}_{k+1})=(\gamma+\frac{1}{\alpha_k})(x_{k+1},y_{k+1})-\frac{1}{\alpha_k}(x_k,y_k),$\\
			$({Z}^\delta_{k+1},{H}^\delta_{k+1})=\delta({Z}_{k+1},{H}_{k+1})+(1-\delta\gamma)(x_{k+1},y_{k+1}),$\\
			$\lambda_{k+1} = \lambda_k + \alpha_k\beta_{k}(A{Z}^\delta_{k+1}+B{H}^\delta_{k+1}-b)$.
		}
	\end{algorithm}

	To better utilize the separable structure of the objective function, we adopt a semi-implicit multiplier scheme to derive the following primal-dual splitting algorithm. We give $\hat{\lambda}_{k+1}$ and $\bar{\lambda}_{k+1}$ through a semi-implicit multiplier scheme as follows:
	\begin{eqnarray}\label{label3}
		\bar{\lambda}_{k+1} = \lambda_k + \alpha_k{\beta}_{k}(AZ^\delta_{k+1}+BH^\delta_{k}-b);\qquad \hat{\lambda}_{k+1} = {\lambda}_{k+1}.
	\end{eqnarray}
	Thus, we can reformulate scheme (\ref{disc}) as Algorithm \ref{al_3}.
	\begin{algorithm}
		\caption{Primal-dual Splitting Algorithm}
		\label{al_3}
		{\bf Initialize:} Let $x_1=x_0=\frac{1}{\gamma} Z_1, y_1=y_0=\frac{1}{\gamma} H_1, {\lambda}_1={\lambda}_0$, $\delta>0,\gamma>0$. \\
		\For{$k = 1, 2,\cdots$}{
			{\bf Step1:} Let $\theta_k=(\alpha_k+{\delta})\beta_k$ and $\eta_{f,k}=\gamma+\frac{1}{\alpha_k}+\mu_f\delta\beta_k,$ \\
			$\widetilde{\lambda}_{1,k} = \lambda_k -\delta\beta_k(Ax_k+By_k-b) +\delta\alpha_k\beta_kB(H_k-\gamma y_k)$,\\
			$\widetilde{x}_k=x_k+\frac{Z_k-\gamma x_k}{\eta_{f,k}},$\\
			$x_{k+1} = \underset{x\in \mathcal{X}}{\arg\min}\left(\mathcal{L}_{\theta_k}\left(x,y_k,\widetilde{\lambda}_{1,k}\right)
			+\frac{\eta_{f,k}}{2\alpha_k\beta_k}\left\|x-\widetilde{x}_k\right\|^2
			+ \frac{\epsilon_k}{2}\left\|x\right\|^2 \right),$\\
			${Z}_{k+1}=\frac{x_{k+1}-x_k}{\alpha_k}+\gamma{x}_{k+1},$\\
			${Z}^\delta_{k+1}=\delta{Z}_{k+1}+(1-\delta\gamma)x_{k+1},$\\
			{\bf Step2:} Let $\eta_{g,k}=\gamma+\frac{1}{\alpha_k}+\mu_g\delta\beta_k,$ \\
			$\widetilde{\lambda}_{2,k} = \lambda_k -\delta\beta_k(Ax_{k}+By_k-b)$,\\
			$\widetilde{y}_k=y_k+\frac{H_k-\gamma y_k}{\eta_{g,k}},$\\
			$y_{k+1} = \underset{y\in \mathcal{Y}}{\arg\min}\left(\mathcal{L}_{\theta_k}\left(x_{k+1},y,\widetilde{\lambda}_{2,k} \right)+\frac{\eta_{g,k}}{2\alpha_k\beta_k}\left\|y-\widetilde{y}_k\right\|^2
			+ \frac{\epsilon_k}{2}\left\|y\right\|^2 \right),$\\
			${H}_{k+1}=\frac{y_{k+1}-y_k}{\alpha_k}+\gamma{y}_{k+1},$\\
			${H}^\delta_{k+1}=\delta{H}_{k+1}+(1-\delta\gamma)y_{k+1},$\\
			{\bf Step3:} $\lambda_{k+1} = \lambda_k + \alpha_k\beta_{k}(A{Z}^\delta_{k+1}+B{H}^\delta_{k+1}-b)$.
		}
	\end{algorithm}

	The following propositions demonstrate that our proposed Algorithms \ref{al_1} and \ref{al_3} can be interpreted as specific instances of the discretization scheme (\ref{disc}), differing only in their choices of multipliers.
	\begin{proposition}\label{pro1}
		Algorithm \ref{al_1} is equivalent to the scheme $(\ref{disc})$ when the multipliers are given implicitly by $(\ref{label1})$.
	\end{proposition}
	\begin{proof}
		By using the optimality criterion, from Step 1 of Algorithm \ref{al_1}, we get
		\begin{align}\label{xjd1}
			\partial_x\mathcal{L}_{\theta_k}(x_{k+1},y_{k+1},\widetilde{\lambda}_k)+\frac{\eta_{f,k}}{\alpha_k\beta_k}(x_{k+1}-\widetilde{x}_k)+\epsilon_kx_{k+1}\ni0.
		\end{align}
		In view of $\lambda_{k+1}=\widetilde{\lambda}_k+(\alpha_k+\delta)\beta_k(Ax_{k+1}+By_{k+1}-b)$, it follows that
		\begin{align*}
			\partial_x \mathcal{L}(x_{k+1},y_{k+1},\lambda_{k+1})&=\partial_x\mathcal{L}(x_{k+1},y_{k+1},\widetilde{\lambda}_k)+
			\theta_kA^T(Ax_{k+1}+By_{k+1}-b)=\partial_x \mathcal{L}_{\theta_k}(x_{k+1},y_{k+1},\widetilde{\lambda}_k).
		\end{align*}
		This together with \eqref{xjd1} implies
		\begin{align*}
			\partial_x \mathcal{L}(x_{k+1},y_{k+1},\lambda_{k+1})+\frac{\eta_{f,k}}{\alpha_k\beta_k}(x_{k+1}-\widetilde{x}_k)+\epsilon_kx_{k+1}\ni0.
		\end{align*}
		There is no doubt that
		\begin{align*}
			x_{k+1}\in\widetilde{x}_k-\frac{\alpha_k\beta_k}{\eta_{f,k}}(\partial f(x_{k+1})+A^T\lambda_{k+1}+\epsilon_kx_{k+1}).
		\end{align*}
		From the definition of $\eta_{f,k}$ and $\widetilde{x}_k$, we have
		\begin{align*}
			\gamma x_{k+1}+\frac{x_{k+1}-x_k}{\alpha_k}-Z_k\in-\mu_f\delta\beta_k(x_{k+1}-x_k)-\alpha_k\beta_k(\partial f(x_{k+1})+A^\top {\lambda}_{k+1}+\epsilon_k{x}_{k+1}).
		\end{align*}
		Using $x_{k+1}-Z_{k+1}^\delta=-\frac{\delta}{\alpha_k}(x_{k+1}-x_k)$ and the definition of $Z_{k+1}$, it follows that
		\begin{eqnarray*}
			\frac{Z_{k+1}-Z_k}{\alpha_k}\in \beta_k\mu_f(x_{k+1}-Z_{k+1}^\delta)-\beta_k(\partial f(x_{k+1})+A^\top {\lambda}_{k+1}+\epsilon_k{x}_{k+1}).
		\end{eqnarray*}
		Similarly, we derive
		\begin{eqnarray*}
			\frac{H_{k+1}-H_k}{\alpha_k}\in \beta_k\mu_g(y_{k+1}-H_{k+1}^\delta)-\beta_k(\partial g(y_{k+1})+B^\top {\lambda}_{k+1}+\epsilon_k{y}_{k+1}).
		\end{eqnarray*}
		Thus, the iterative sequence $ \{(x_{k},y_{k},\lambda_{k})\}_{k\geq1} $ generated by Algorithm \ref{al_1} satisfies scheme (\ref{disc}). Since the above process are invertible, Algorithm \ref{al_1} can conversely be derived from (\ref{disc}).
	\end{proof}
	\begin{proposition}\label{pro4}
		Algorithm \ref{al_3} is equivalent to the scheme $(\ref{disc})$ when the multipliers are given semi-implicitly by $(\ref{label3})$.
	\end{proposition}
	\begin{proof}
		As $\widetilde{\lambda}_{1,k} = \lambda_k -\delta\beta_k(Ax_k+By_k-b) +\delta\alpha_k\beta_kB(H_k-\gamma y_k)$, we can observe that
		\begin{eqnarray*}
			\bar{\lambda}_{k+1}=\widetilde{\lambda}_{1,k}+(\alpha_k+\delta)\beta_k(Ax_{k+1}+By_k-b),
		\end{eqnarray*}
		Then,
		\begin{equation}\label{xjd2}
			\partial _x \mathcal{L}(x_{k+1},y_{k+1},\bar \lambda_{k+1})
			=\partial _x \mathcal{L}( x_{k+1},y_{k+1},\widetilde\lambda_{1,k})+(\alpha_k+\delta)\beta_kA^T(Ax_{k+1}+By_k-b) 
			=\partial _x\mathcal{L}_{\theta_k}(x_{k+1},y_k,\widetilde{\lambda}_{1,k}).
		\end{equation}
		It follows from the optimality condition for Step 1 in Algorithm \ref{al_3} that
		\begin{eqnarray}\label{xjd3}
			\partial_x\mathcal{L}_{\theta_k}(x_{k+1},y_k,\widetilde{\lambda}_{1,k})+\frac{\eta_{f,k}}{\alpha_k\beta_k}(x_{k+1}-\widetilde{x}_k)+\epsilon_kx_{k+1}\ni0.
		\end{eqnarray}
		Now, combining \eqref{xjd3} with \eqref{xjd2}, we obtain
		\begin{eqnarray*}
			\partial_x \mathcal{L}(x_{k+1},y_k,\bar{\lambda}_{k+1})+\frac{\eta_{f,k}}{\alpha_k\beta_k}(x_{k+1}-\widetilde{x}_k)+\epsilon_kx_{k+1}\ni0.
		\end{eqnarray*}
		This implies
		\begin{eqnarray*}
			x_{k+1}\in\widetilde{x}_k-\frac{\alpha_k\beta_k}{\eta_{f,k}}(\partial f(x_{k+1})+A^T\bar\lambda_{k+1}+\epsilon_kx_{k+1}).
		\end{eqnarray*}
		Thus, using a similar argument as Proposition \ref{pro1} but with $\lambda_{k+1}$ replaced by $\bar\lambda_{k+1}$, we know that
		\begin{eqnarray*}
			\frac{Z_{k+1}-Z_k}{\alpha_k}\in \beta_k\mu_f(x_{k+1}-Z_{k+1}^\delta)-\beta_k(\partial f(x_{k+1})+A^\top \bar{\lambda}_{k+1}+\epsilon_k{x}_{k+1}).
		\end{eqnarray*}
		From $\widetilde{\lambda}_{2,k} = \lambda_k -\delta\beta_k(Ax_{k}+By_k-b)$, we get
		\begin{eqnarray*}
			{\lambda}_{k+1}=\widetilde{\lambda}_{2,k}+(\alpha_k+\delta)\beta_k(Ax_{k+1}+By_{k+1}-b).
		\end{eqnarray*}
		Thus,
		\begin{align*}
			\partial _y \mathcal{L}(x_{k+1},y_{k+1}, \lambda_{k+1})&=\partial _y\mathcal{L}(x_{k+1},y_{k+1},\widetilde{\lambda}_{2,k})+
			\theta_kA^T(Ax_{k+1}+By_{k+1}-b)=\partial _y\mathcal{L}_{\theta_k}(x_{k+1},y_{k+1},\widetilde{\lambda}_{2,k}).
		\end{align*}
		Using the optimality condition for Step 2 in Algorithm \ref{al_3}, we obtain
		\begin{eqnarray*}
			\partial_y\mathcal{L}_{\theta_k}(x_{k+1},y_{k+1},\widetilde{\lambda}_{2,k})+\frac{\eta_{g,k}}{\alpha_k\beta_k}(y_{k+1}-\widetilde{y}_k)+\epsilon_ky_{k+1}\ni0.
		\end{eqnarray*}
		This leads to
		\begin{eqnarray*}
			\partial_y\mathcal{L}(x_{k+1},y_{k+1},{\lambda}_{k+1})+\frac{\eta_{g,k}}{\alpha_k\beta_k}(y_{k+1}-\widetilde{y}_k)+\epsilon_ky_{k+1}\ni0.
		\end{eqnarray*}
		It follows that
		\begin{eqnarray*}
			y_{k+1}\in\widetilde{y}_k-\frac{\alpha_k\beta_k}{\eta_{g,k}}(\partial g(y_{k+1})+B^T\lambda_{k+1}+\epsilon_ky_{k+1}).
		\end{eqnarray*}
		Following the proof of Proposition \ref{pro1}, we can establish the equivalence between the algorithm \ref{al_3} and the scheme (\ref{disc}).
	\end{proof}
	\begin{remark}
		Given another semi-implicit multiplier scheme below:
		\begin{eqnarray*}\label{labelfd}
			\bar{\lambda}_{k+1} = {\lambda}_{k+1};\qquad \hat{\lambda}_{k+1} = \lambda_k + \alpha_k{\beta}_{k}(AZ^\delta_{k}+BH^\delta_{k+1}-b).
		\end{eqnarray*}
		we can obtain a splitting algorithm symmetric to Algorithm \ref{al_3}, essentially swapping the two variables and their corresponding coefficients. Hence, we only analyze Algorithm \ref{al_3} here. The interested reader is referred to Remark 3.2 in \cite{30}.
	\end{remark}
	
	\section{Fast convergence properties}
	
	In this section, we will analyze the convergence rates of our proposed algorithms. Before that, we recall the following equality
	\begin{equation}\label{ds}
		2\langle a,b\rangle=\|a+b\|^2-\|a\|^2-\|b\|^2.
	\end{equation}
	For any proper, closed and lower semi-continuous convex function $f$ on $\mathcal X$, we write $f\in S^0_{\mu}(\mathcal X)$ with $\mu \geq 0$ if
	\begin{equation*}
		f(x_1)-f(x_2)-\langle {w},x_1-x_2\rangle\geq\frac{\mu}{2}\|x_1-x_2\|^2,\qquad\forall(x_1,x_2)\in \mathcal {X\times X},
	\end{equation*}
	where ${w}\in\partial f(x_2)$. The function $f$ is convex when $\mu=0$ and is $\mu$-strongly convex when $\mu>0$. Besides, we denote
	\begin{equation*}
		\mathcal{L}^{\epsilon} (x,y,{\lambda^*}):= \mathcal{L}(x,y,{\lambda^*})+\frac{{\epsilon}}{2}\left\|x\right\|^2+\frac{{\epsilon}}{2}\left\|y\right\|^2.
	\end{equation*}
	If $f\in S^0_{\mu_f}(\mathcal X)$ and $g\in S^0_{\mu_g}(\mathcal Y)$ with   ${\mu}_f, {\mu}_g\geq 0$, then we have
	\begin{equation}\label{q}
		\frac{{\mu}_f}{2}\|x_1-x_2\|^2+\frac{{\mu}_g}{2}\|y_1-y_2\|^2\leq {\mathcal{L}^\epsilon(x_1,y_1,{\lambda^*})-\mathcal{L}^\epsilon(x_2,y_2,{\lambda^*})-\langle p,x_1-x_2\rangle-\langle q, y_1-y_2\rangle},
	\end{equation}
	where $p\in\partial_x \mathcal{L}^\epsilon(x_2,y_2,{\lambda^*})$ and $q\in\partial_y \mathcal{L}^\epsilon(x_2,y_2,{\lambda^*})$.

	We also recall the lemma below, which is essential for proving the convergence rates.
	\begin{lemma}\label{A.1}\textup{\cite[Lemma 4]{9}}
		Let $\{h_k\}_{k\geq 1}$ be a sequence of vectors in $\mathbb{R}^{n}$ and $\{b_k\}_{k\geq 1}$ be a sequence in $[0,1)$. Suppose
		\[\left\|h_{k+1}+\sum_{i=1}^{k} b_{i} h_{i}\right\|\leq C, \quad\forall k\geq 1.\]
		Then,
		\[ \sup_{k\geq 1}\|h_k\|\leq +\infty.\]
	\end{lemma}

	Throughout this paper, we introduce the following assumptions regarding the parameters of the proposed algorithms.
	\begin{assumption}\label{ass}
		Let $f\in\mathcal S_{\mu_f}^0(\mathcal X)$ with $\mu_f\geq  0$ and $g\in\mathcal S_{\mu_g}^0(\mathcal Y)$ with $\mu_g\geq  0$. $\{{\beta}_k\}_{k\geq 1}$ is a positive and nondecreasing sequence and $\lim_{k\rightarrow{+\infty}}\beta_{k}={+\infty}$. $\{{\alpha}_k\}_{k\geq 1}$ is a positive sequence. For every $k\geq1$, the parameters $\delta,\gamma$ and these two sequences satisfy
		\begin{eqnarray*}
			{\delta\gamma}-1\geq0,\qquad\delta{\beta}_{k+1}\leq\delta{\beta}_{k}+\alpha_k{\beta}_{k}.
		\end{eqnarray*}
	\end{assumption}
	\begin{assumption}\label{ass4}
		Let $f\in\mathcal S_{\mu_f}^0(\mathcal X)$ with $\mu_f\geq  0$ and $g\in\mathcal S_{\mu_g}^0(\mathcal Y)$ with $\mu_g \geq 0$. $\{{\beta}_k\}_{k\geq 1}$ is a positive and nondecreasing sequence and $\lim_{k\rightarrow{+\infty}}\beta_{k}={+\infty}$. $\{{\alpha}_k\}_{k\geq 1}$ is a positive sequence. For every $k\geq 1$, the parameters $\delta,\gamma$ and these two sequences satisfy
		\begin{eqnarray*}
			&{\delta\gamma}-1\geq 0,\qquad\delta{\beta}_{k+1}\leq\delta{\beta}_{k}+\alpha_k{\beta}_{k},\nonumber\\
			&\|B\|^2(\alpha_{k+1}^2\beta_{k+1}^2-\alpha_k^2\beta_k^2) \leq \alpha_k\beta_k\mu_g.
		\end{eqnarray*}
	\end{assumption}

	Now, we employ a unified energy function to examine sequences satisfying (\ref{disc}):
	\begin{eqnarray*}
		\mathcal{E}_k =I_k^1 +I_k^2+I_k^3+I_k^4
	\end{eqnarray*}
	with
	\begin{equation}\label{energy}
		\begin{cases}{}
			I_k^1 = {\delta}^2{\beta}_k(\mathcal{L}(x_k,y_k,\lambda^*)-\mathcal{L}(x^*,y^*,\lambda^*)+\frac{{\epsilon}_k}{2}\left\|x_k\right\|^2
			+\frac{{\epsilon}_k}{2}\left\|y_k\right\|^2), \\
			I_k^2 = \frac{1}{2}\|{{Z}^{\delta}_{k}}-x^*\|^2+\frac{1}{2}\|{{H}^{\delta}_{k}}-y^*\|^2, \\
			I_k^3 = \frac{\delta\gamma-1}{2}\|x_k-x^*\|^2+\frac{\delta\gamma-1}{2}\|y_k-y^*\|^2,  \\
			I_k^4 = \frac{\delta}{2}\|\lambda_k-\lambda^*\|^2.
		\end{cases}
	\end{equation}
	
	Clearly, $\mathcal{E}_k\geq 0$ for all $(x^*,\lambda^*)\in \mathbb S$ and $k\geq 1$. To estimate the convergence rates, we require the following lemma.
	\begin{lemma}\label{le1}
		Suppose that $f\in\mathcal S_{\mu_f}^0(\mathcal X)$ with $\mu_f\geq  0$ and $g\in\mathcal S_{\mu_g}^0(\mathcal Y)$ with $\mu_g\geq  0$. Let $ \{(x_{k},y_{k},\lambda_{k})\}_{k\geq1} $ be the sequence generated by the discretization $(\ref{disc})$ and $(x^*,y^*,{\lambda}^*)\in \Omega$. For every $k\geq1$, it holds
		\begin{equation}\label{energybds}
			\begin{aligned}
				\mathcal{E}_{k+1}-\mathcal{E}_k \leq& ({\delta}^2{\beta}_{k+1}-{\delta}^2{\beta}_{k}-{\delta}\alpha_k{\beta}_{k})(\mathcal{L}(x_{k+1},y_{k+1},\lambda^*)-\mathcal{L}(x^*,y^*,\lambda^*))\\
				&+\frac{{\delta}\alpha_k{\beta}_k{\epsilon}_k}{2}(\|x^*\|^2+\|y^*\|^2)-\frac{\delta}{2}\|{\lambda}_{k+1}-{\lambda}_{k}\|^2\\
				&+(\frac{{\delta}^2{\beta}_{k+1}{\epsilon}_{k+1}}{2}
				-\frac{{\delta}^2{\beta}_k{\epsilon}_k}{2}-\frac{{\delta}\alpha_k{\beta}_k{\epsilon}_k}{2})(\|x_{k+1}\|^2+\|y_{k+1}\|^2)\\
				&-\frac{1}{2}\|Z_{k+1}^\delta-Z_{k}^\delta\|^2-\frac{1}{2}\|H_{k+1}^\delta-H_{k}^\delta\|^2 -\frac{{\delta}^2{\beta}_k{\mu}_f}{2}\|x_{k+1}-x_k\|^2-\frac{{\delta}^2{\beta}_k{\mu}_g}{2}\|y_{k+1}-y_k\|^2\\
				&-\frac{\delta\alpha_k\beta_k\mu_f}{2}\|Z_{k+1}^\delta-x^*\|^2-\frac{\delta\alpha_k\beta_k\mu_g}{2}\|H_{k+1}^\delta-y^*\|^2\\
				&-\frac{\delta\alpha_k\beta_k\mu_f}{2}\|x_{k+1}-Z_{k+1}^\delta\|^2-\frac{\delta\alpha_k\beta_k\mu_g}{2}\|y_{k+1}-H_{k+1}^\delta\|^2\\
				&-(\delta\gamma-1)(\frac{\delta}{a_k}+\frac{1}{2})\|x_{k+1}-x_k\|^2-(\delta\gamma-1)(\frac{\delta}{a_k}+\frac{1}{2})\|y_{k+1}-y_k\|^2\\
				&-\delta\alpha_k\beta_k\langle \bar\lambda_{k+1}-\lambda^*,A(Z_{k+1}^\delta-x^*) \rangle-\delta\alpha_k\beta_k\langle \hat\lambda_{k+1}-\lambda^*,B(H_{k+1}^\delta-y^*) \rangle\\
				&+\delta\langle\lambda_{k+1}-\lambda_k,\lambda_{k+1}-\lambda^*\rangle.
			\end{aligned}
		\end{equation}
	\end{lemma}
	\begin{proof}
		By the definition of $I^1_k$, we have
		\begin{equation}\label{energybds1}
			\begin{aligned}
				I_{k+1}^1-I_k^1 =& {\delta}^2{\beta}_{k+1}(\mathcal{L}(x_{k+1},y_{k+1},\lambda^*)-\mathcal{L}(x^*,y^*,\lambda^*))\\
				&+\frac{{\delta}^2{\beta}_{k+1}{\epsilon}_{k+1}}{2}(\|x_{k+1}\|^2+\|y_{k+1}\|^2)
				-\frac{{\delta}^2{\beta}_{k}{\epsilon}_{k}}{2}(\|x_{k}\|^2+\|y_{k}\|^2)\\
				&-{\delta}^2{\beta}_{k}(\mathcal{L}(x_{k},y_{k},\lambda^*)-\mathcal{L}(x^*,y^*,\lambda^*)).
			\end{aligned}
		\end{equation}
		From (\ref{disc}), we can derive
		\begin{align*}
			\frac{{Z}_{k+1}-{Z}_{k}}{{\alpha_k\beta}_k}&\in -\partial f(x_{k+1})-A^\top \bar\lambda_{k+1}-{\epsilon}_{k}x_{k+1}
			+\mu_f(x_{k+1}-Z_{k+1}^\delta)\\
			&=-\partial_x L^{{\epsilon}_{k}}(x_{k+1},y_{k+1},\lambda^*)-A^\top(\bar{\lambda}_{k+1}-\lambda^*)
			+\mu_f(x_{k+1}-Z_{k+1}^\delta),
		\end{align*}
		and
		\begin{align*}
			\frac{{H}_{k+1}-{H}_{k}}{\alpha_k{\beta}_k}&\in -\partial g(y_{k+1})-B^\top \hat\lambda_{k+1}-{\epsilon}_{k}y_{k+1}
			+\mu_g(y_{k+1}-H_{k+1}^\delta)\\
			&=-\partial_y L^{\epsilon_k}(x_{k+1},y_{k+1},\lambda^*)-B^\top(\hat{\lambda}_{k+1}-\lambda^*)+\mu_g(y_{k+1}-H_{k+1}^\delta).
		\end{align*}
		Denote
		\begin{equation*}\label{M}
			M_k=-A^\top(\bar{\lambda}_{k+1}-\lambda^*)-\frac{{Z}_{k+1}-{Z}_{k}}{\alpha_k{\beta}_k}+\mu_f(x_{k+1}-Z_{k+1}^\delta) \in \partial_x \mathcal{L}^{\epsilon_k}(x_{k+1},y_{k+1},\lambda^*),
		\end{equation*}
		and
		\begin{equation*}\label{N}
			N_k=-B^\top(\hat{\lambda}_{k+1}-\lambda^*)-\frac{{H}_{k+1}-{H}_{k}}{\alpha_k{\beta}_k}+\mu_g(y_{k+1}-H_{k+1}^\delta) \in \partial_y \mathcal{L}^{\epsilon_k}(x_{k+1},y_{k+1},\lambda^*).
		\end{equation*}
		From (\ref{ds}), we get
		\begin{align*}
			I_{k+1}^2-I_k^2 =& \frac{1}{2}\|{{Z}^{\delta}_{k+1}}-x^*\|^2-\frac{1}{2}\|{{Z}^{\delta}_{k}}-x^*\|^2+\frac{1}{2}\|{{H}^{\delta}_{k+1}}-y^*\|^2
			-\frac{1}{2}\|{{H}^{\delta}_{k}}-y^*\|^2\nonumber\\
			=&{\delta}\langle Z_{k+1}-Z_k,{Z}^{\delta}_{k+1}-x^*\rangle+(1-\delta\gamma)\langle x_{k+1}-x_k,{Z}^{\delta}_{k+1}-x^*\rangle\nonumber\\
			&+{\delta}\langle H_{k+1}-H_k,{H}^{\delta}_{k+1}-x^*\rangle+(1-\delta\gamma)\langle y_{k+1}-y_k,{H}^{\delta}_{k+1}-y^*\rangle\nonumber\\
			&-\frac{1}{2}\|Z_{k+1}^\delta-Z_{k}^\delta\|^2-\frac{1}{2}\|H_{k+1}^\delta-H_{k}^\delta\|^2.
		\end{align*}
		In details, we know that
		\begin{equation}\label{energybds2.1}
			\begin{aligned}
				&{\delta}\langle Z_{k+1}-Z_k,{Z}^{\delta}_{k+1}-x^*\rangle + {\delta}\langle H_{k+1}-H_k,{H}^{\delta}_{k+1}-x^*\rangle\\
				=& -\delta\alpha_k\beta_k\langle \bar\lambda_{k+1}-\lambda^*,A(Z_{k+1}^\delta-x^*) \rangle
				-\delta\alpha_k\beta_k\langle \hat\lambda_{k+1}-\lambda^*,B(H_{k+1}^\delta-y^*) \rangle\\
				&+\delta\alpha_k\beta_k\mu_f\langle x_{k+1}-Z_{k+1}^\delta,Z_{k+1}^\delta-x^* \rangle +\delta\alpha_k\beta_k\mu_g\langle y_{k+1}-H_{k+1}^\delta,H_{k+1}^\delta-y^* \rangle \\
				&-\delta\alpha_k{\beta}_k\langle M_k, x_{k+1}-x^*\rangle-{\delta}^2{\beta}_k\langle M_k, x_{k+1}-x_{k}\rangle\\ &-\delta\alpha_k{\beta}_k\langle N_k, y_{k+1}-y^*\rangle-{\delta}^2{\beta}_k\langle N_k, y_{k+1}-y_{k}\rangle\\
				=&-\delta\alpha_k\beta_k\langle \bar\lambda_{k+1}-\lambda^*,A(Z_{k+1}^\delta-x^*) \rangle
				-\delta\alpha_k\beta_k\langle \hat\lambda_{k+1}-\lambda^*,B(H_{k+1}^\delta-y^*) \rangle\\
				&-\delta\alpha_k{\beta}_k\langle M_k, x_{k+1}-x^*\rangle-{\delta}^2{\beta}_k\langle M_k, x_{k+1}-x_{k}\rangle\\ &-\delta\alpha_k{\beta}_k\langle N_k, y_{k+1}-y^*\rangle-{\delta}^2{\beta}_k\langle N_k, y_{k+1}-y_{k}\rangle\\
				&+\frac{\delta\alpha_k\beta_k\mu_f}{2}\|x_{k+1}-x^*\|^2-\frac{\delta\alpha_k\beta_k\mu_f}{2}\|Z_{k+1}^\delta-x^*\|^2\\
				&-\frac{\delta\alpha_k\beta_k\mu_f}{2}\|x_{k+1}-Z_{k+1}^\delta\|^2
				+\frac{\delta\alpha_k\beta_k\mu_g}{2}\|y_{k+1}-y^*\|^2\\
				&-\frac{\delta\alpha_k\beta_k\mu_g}{2}\|H_{k+1}^\delta-y^*\|^2-\frac{\delta\alpha_k\beta_k\mu_g}{2}\|y_{k+1}-H_{k+1}^\delta\|^2,
			\end{aligned}
		\end{equation}
		and
		\begin{equation}\label{energybds2.2}
			\begin{aligned}
				&(1-\delta\gamma)\langle x_{k+1}-x_k,{Z}^{\delta}_{k+1}-x^*\rangle+(1-\delta\gamma)\langle y_{k+1}-y_k,{H}^{\delta}_{k+1}-y^*\rangle\\
				=&(1-\delta\gamma)\langle x_{k+1}-x_k,x_{k+1}-x^*\rangle + \frac{(1-\delta\gamma)\delta}{\alpha_k}\langle x_{k+1}-x_k,x_{k+1}-x_k\rangle\\
				&+(1-\delta\gamma)\langle y_{k+1}-y_k,y_{k+1}-y^*\rangle + \frac{(1-\delta\gamma)\delta}{\alpha_k}\langle y_{k+1}-y_k,y_{k+1}-y_k\rangle\\
				=&-\frac{\delta\gamma-1}{2}(\|x_{k+1}-x^*\|^2-\|x_k-x^*\|^2)-(\delta\gamma-1)(\frac{\delta}{\alpha_k}+\frac{1}{2})\|x_{k+1}-x_k\|^2\\
				&-\frac{\delta\gamma-1}{2}(\|y_{k+1}-y^*\|^2-\|y_k-y^*\|^2)-(\delta\gamma-1)(\frac{\delta}{\alpha_k}+\frac{1}{2})\|y_{k+1}-y_k\|^2.
			\end{aligned}
		\end{equation}
		Further, using the equation $(\ref{q})$, we have
		\begin{equation}\label{energybds2.3}
			\begin{aligned}
				&-\delta\alpha_k{\beta}_k\langle M_k,x_{k+1}-x^*\rangle-\delta\alpha_k{\beta}_k\langle N_k,y_{k+1}-y^*\rangle\\
				\leq&\delta\alpha_k{\beta}_k(\mathcal{L}^{\epsilon_k}(x^*,y^*,\lambda^*)
				-\mathcal{L}^{\epsilon_k}(x_{k+1},y_{k+1},\lambda^*)-\frac{\mu_f}{2}\|x_{k+1}-x^*\|^2-\frac{\mu_g}{2}\|y_{k+1}-y^*\|^2),
			\end{aligned}
		\end{equation}
		and
		\begin{equation}\label{energybds2.4}
			\begin{aligned}
				&-\delta^2{\beta}_k\langle M_k,x_{k+1}-x_k\rangle-\delta^2{\beta}_k\langle N_k,y_{k+1}-y_k\rangle\\
				\leq&\delta^2{\beta}_k(\mathcal{L}^{\epsilon_k}(x_k,y_k,\lambda^*)-\mathcal{L}^{\epsilon_k}(x_{k+1},y_{k+1},\lambda^*)
				-\frac{\mu_f}{2}\|x_{k+1}-x_k\|^2-\frac{\mu_g}{2}\|y_{k+1}-y_k\|^2),
			\end{aligned}
		\end{equation}
		where $\mu_f \geq 0$ and $\mu_g \geq 0$. Combining (\ref{energybds2.1}) , (\ref{energybds2.2}) , (\ref{energybds2.3}) and (\ref{energybds2.4}), it holds
		\begin{equation}\label{energybds2z}
			\begin{aligned}
				&I_{k+1}^2-I_k^2 \\
				\leq&
				-\delta\alpha_k\beta_k\langle \bar\lambda_{k+1}-\lambda^*,A(Z_{k+1}^\delta-x^*) \rangle
				-\delta\alpha_k\beta_k\langle \hat\lambda_{k+1}-\lambda^*,B(H_{k+1}^\delta-y^*) \rangle\\
				&-\frac{\delta\gamma-1}{2}(\|x_{k+1}-x^*\|^2-\|x_k-x^*\|^2)-(\delta\gamma-1)(\frac{\delta}{\alpha_k}+\frac{1}{2})\|x_{k+1}-x_k\|^2\\
				&-\frac{\delta\gamma-1}{2}(\|y_{k+1}-y^*\|^2-\|y_k-y^*\|^2)-(\delta\gamma-1)(\frac{\delta}{\alpha_k}+\frac{1}{2})\|y_{k+1}-y_k\|^2\\
				&+\delta\alpha_k{\beta}_k(\mathcal{L}^{\epsilon_k}(x^*,y^*,\lambda^*)
				-\mathcal{L}^{\epsilon_k}(x_{k+1},y_{k+1},\lambda^*))\\
				&+\delta^2{\beta}_k(\mathcal{L}^{\epsilon_k}(x_k,y_k,\lambda^*)-\mathcal{L}^{\epsilon_k}(x_{k+1},y_{k+1},\lambda^*))\\
				&-\frac{1}{2}\|Z_{k+1}^\delta-Z_{k}^\delta\|^2-\frac{1}{2}\|H_{k+1}^\delta-H_{k}^\delta\|^2-\frac{{\delta}^2{\beta}_k{\mu}_f}{2}\|x_{k+1}-x_k\|^2\\
				&-\frac{{\delta}^2{\beta}_k{\mu}_g}{2}\|y_{k+1}-y_k\|^2-\frac{\delta\alpha_k\beta_k\mu_f}{2}\|Z_{k+1}^\delta-x^*\|^2-\frac{\delta\alpha_k\beta_k\mu_g}{2}\|H_{k+1}^\delta-y^*\|^2\\
				&-\frac{\delta\alpha_k\beta_k\mu_f}{2}\|x_{k+1}-Z_{k+1}^\delta\|^2-\frac{\delta\alpha_k\beta_k\mu_g}{2}\|y_{k+1}-H_{k+1}^\delta\|^2,
			\end{aligned}
		\end{equation}
		Then we can obtain
		\begin{equation}\label{energybds3}
			\begin{aligned}
				I_{k+1}^3-I_k^3 =& \frac{\delta\gamma-1}{2}\|{x}_{k+1}-x^*\|^2-\frac{\delta\gamma-1}{2}\|{x}_{k}-x^*\|^2 \\
				&+\frac{\delta\gamma-1}{2}\|{y}_{k+1}-y^*\|^2-\frac{\delta\gamma-1}{2}\|{y}_{k}-y^*\|^2.
			\end{aligned}
		\end{equation}
		We deduce from (\ref{ds}) that
		\begin{equation}\label{energybds4}
			\begin{aligned}
				I_{k+1}^4-I_k^4 =& \frac{\delta}{2}\|{\lambda}_{k+1}-\lambda^*\|^2-\frac{\delta}{2}\|{\lambda}_{k}-\lambda^*\|^2 \\
				=&\delta\langle {\lambda}_{k+1}-{\lambda}_{k},{\lambda}_{k+1}-\lambda^*\rangle-\frac{\delta}{2}\|{\lambda}_{k+1}-{\lambda}_{k}\|^2.
			\end{aligned}
		\end{equation}
		By summing (\ref{energybds1}), (\ref{energybds2z}), (\ref{energybds3}) and (\ref{energybds4}), we complete the proof of this lemma.
	\end{proof}

	From the unified estimate (\ref{energybds}), it is evident that setting $\hat{\lambda}_{k+1} = \bar{\lambda}_{k+1} = {\lambda}_{k+1}$ ensures a contraction.
	\begin{theorem}\label{th3.1}
		Suppose that Assumption \ref{ass} holds. Let $\{(x_k,y_k,\lambda_k)\}_{k\geq 1}$ be the sequence generated by Algorithm \ref{al_1} and  $(x^*,y^*,\lambda^*)\in\Omega$.  Assume that $ \{\epsilon_{k}\}_{k\geq1} $ is a nonincreasing sequence and $\sum_{k=1}^{+\infty}\alpha_k\beta_{k}\epsilon_{k}<{+\infty}$, then we have the sequence $ \{(x_{k},y_k,\lambda_k)\}_{k\geq1} $ is bounded and the following statements:
		\begin{eqnarray*}
			\begin{aligned}
				&\mathcal{L} ( x_k, y_k,\lambda ^*)-\mathcal{L} ( x^*,y^*,\lambda ^* )=\mathcal{O}\left( \frac{1}{\beta_k} \right),\\
				|\varPhi ( x_k,y_k)-&\varPhi( x^*,y^*)|=\mathcal{O} \left( \frac{1}{\beta_k} \right),\,\,\,\ \left \| Ax_k+B y_k-b\right \|=\mathcal{O}\left ( \frac{1}{\beta_k}\right ).\\
			\end{aligned}
		\end{eqnarray*}
	\end{theorem}
	\begin{proof}
		From Step 2 of Algorithm \ref{al_1} that
		\begin{align*}
			&\delta\alpha_k\beta_k\langle \lambda_{k+1}-\lambda^*,A(Z_{k+1}^\delta-x^*) \rangle+\delta\alpha_k\beta_k\langle \lambda_{k+1}-\lambda^*,B(H_{k+1}^\delta-y^*) \rangle\nonumber\\
			=&\delta\langle\lambda_{k+1}-\lambda_k,\lambda_{k+1}-\lambda^*\rangle.
		\end{align*}
		As before, we calculate $\mathcal{E}_{k+1}-\mathcal{E}_k$ in (\ref{energybds}), then we have
		\begin{equation}\label{energyt1.1}
			\begin{aligned}
				\mathcal{E}_{k+1}-\mathcal{E}_k \leq& ({\delta}^2{\beta}_{k+1}-{\delta}^2{\beta}_{k}-{\delta}\alpha_k{\beta}_{k})(\mathcal{L}(x_{k+1},y_{k+1},\lambda^*)-\mathcal{L}(x^*,y^*,\lambda^*))\\
				&+\frac{{\delta}\alpha_k{\beta}_k{\epsilon}_k}{2}(\|x^*\|^2+\|y^*\|^2)-\frac{\delta}{2}\|{\lambda}_{k+1}-{\lambda}_{k}\|^2\\
				&+(\frac{{\delta}^2{\beta}_{k+1}{\epsilon}_{k+1}}{2}
				-\frac{{\delta}^2{\beta}_k{\epsilon}_k}{2}-\frac{{\delta}\alpha_k{\beta}_k{\epsilon}_k}{2})(\|x_{k+1}\|^2+\|y_{k+1}\|^2)\\
				&-\frac{1}{2}\|Z_{k+1}^\delta-Z_{k}^\delta\|^2-\frac{1}{2}\|H_{k+1}^\delta-H_{k}^\delta\|^2\\
				&-(\delta\gamma-1)(\frac{\delta}{a_k}+\frac{1}{2})\|x_{k+1}-x_k\|^2-(\delta\gamma-1)(\frac{\delta}{a_k}+\frac{1}{2})\|y_{k+1}-y_k\|^2\\
				&-\frac{{\delta}^2{\beta}_k{\mu}_f}{2}\|x_{k+1}-x_k\|^2-\frac{{\delta}^2{\beta}_k{\mu}_g}{2}\|y_{k+1}-y_k\|^2\\
				&-\frac{\delta\alpha_k\beta_k\mu_f}{2}\|Z_{k+1}^\delta-x^*\|^2-\frac{\delta\alpha_k\beta_k\mu_g}{2}\|H_{k+1}^\delta-y^*\|^2\\
				&-\frac{\delta\alpha_k\beta_k\mu_f}{2}\|x_{k+1}-Z_{k+1}^\delta\|^2-\frac{\delta\alpha_k\beta_k\mu_g}{2}\|y_{k+1}-H_{k+1}^\delta\|^2.
			\end{aligned}
		\end{equation}
		By Assumption \ref{ass} and $\{\epsilon_{k}\}_{k\geq1}$ is a nonincreasing sequence, we get
		\begin{eqnarray*}
			{{\delta}^2{\beta}_{k+1}{\epsilon}_{k+1}}
			-{{\delta}^2{\beta}_k{\epsilon}_k}-{{\delta}\alpha_k{\beta}_k{\epsilon}_k}\leq ({\delta}^2{\beta}_{k+1}-{\delta}^2{\beta}_{k}-{\delta}\alpha_k{\beta}_{k})\epsilon_k\leq 0.
		\end{eqnarray*}
		It follows from (\ref{energyt1.1}) that
		\begin{eqnarray}\label{energyt1.3}
			\mathcal{E}_{k+1}-\mathcal{E}_k \leq \frac{{\delta}\alpha_k{\beta}_k{\epsilon}_k}{2}(\|x^*\|^2+\|y^*\|^2), \forall{k\geq1},
		\end{eqnarray}
		which together with $\sum_{k=1}^{+\infty}\alpha_k\beta_{k}\epsilon_{k}<{+\infty}$ leads to
		\begin{eqnarray*}
			\mathcal{E}_{k}\leq \mathcal{E}_1 + C_1, \forall{k\geq1},
		\end{eqnarray*}
		where $C_1$ is a positive constant. Again using (\ref{energy}), it follows that
		\begin{align*}
			\mathcal{L}(x_k,&y_k,\lambda^*)-\mathcal{L}(x^*,y^*,\lambda^*)\leq \frac{\mathcal{E}_1+C_1}{\delta^2\beta_k},\,\,\,\|\lambda_{k}-\lambda^*\|^2\leq\frac{2(\mathcal{E}_1+C_1)}{\delta},\nonumber\\
			&\|x_k-x^*\|^2\leq \frac{2(\mathcal{E}_1+C_1)}{\delta\gamma-1},\,\,\,\|y_k-y^*\|^2\leq \frac{2(\mathcal{E}_1+C_1)}{\delta\gamma-1}.
		\end{align*}
		Thus, the sequence $ \{(x_{k},y_k,\lambda_k)\}_{k\geq1} $ is bounded and
		\begin{eqnarray*}
			\mathcal{L} ( x_k, y_k,\lambda ^*)-\mathcal{L} ( x^*,y^*,\lambda ^* )=\mathcal{O}\left( \frac{1}{\beta_k} \right).
		\end{eqnarray*}
		By Step 2 of Algorithm $\ref{al_1}$, we get
		\begin{align*}
			\lambda_{k+1}-\lambda_1 &= \sum_{i=1}^{k}(\lambda_{i+1}-\lambda_i)\nonumber\\
			&=\sum_{i=1}^{k}\alpha_i{\beta}_i(A{{Z}^{\delta}_{i+1}}+B{{H}^{\delta}_{i+1}}-b)\nonumber\\
			&=\sum_{i=1}^{k}\alpha_i{\beta}_i(1+\frac{\delta}{\alpha_i})(Ax_{i+1}+By_{i+1}-b)-\delta{\beta}_i(Ax_i+By_i-b).
		\end{align*}
		For notation simplicity, denote $h_k:=\beta_{k-1}(\alpha_{k-1}+\delta)(Ax_k+By_k-b)$ and $b_k:=1-\frac{\delta\beta_k}{(\alpha_{k-1}+\delta)\beta_{k-1}}$. It is easy to know that $b_k\in[0,1)$. Then it yields
		\begin{align*}
			\lambda_{k+1}-\lambda_1 &= \sum_{i=1}^{k}h_{i+1}-\frac{\delta\beta_i}{(\alpha_{i-1}+\delta)\beta_{i-1}}h_i\nonumber\\
			&=\sum_{i=1}^{k}h_{i+1}-h_i+\sum_{i=1}^{k}b_i h_i\nonumber\\
			&=h_{k+1}-h_1+\sum_{i=1}^{k}b_i h_i.
		\end{align*}
		This together with the boundedness of $\{\lambda_k\}_{k\geq1}$ yields
		\begin{eqnarray*}
			\|h_{k+1}+\sum_{i=1}^{k}b_ih_i\|=\|\lambda_{k+1}-\lambda_1\|+\|h_1\|\leq C_2, \forall k\geq1,
		\end{eqnarray*}
		where $C_2$ is a nonnegative constant. Applying Lemma \ref{A.1}, we have
		\begin{eqnarray*}
			\|Ax_k+By_k-b\|=\mathcal{O}\left( \frac{1}{\beta_k} \right).
		\end{eqnarray*}
		Therefore, we obtain
		\begin{eqnarray*}
			|{~\varPhi (x_k,y_k)}-{~\varPhi (x^*,y^*)}|\leq \mathcal{L}(x_k,y_k,\lambda^*)-\mathcal{L}(x^*,y^*,\lambda^*)+\|\lambda^*\|\|Ax_k+By_k-b\|.
		\end{eqnarray*}
		Consequently, this gives
		\begin{eqnarray*}\label{energyt1.14}
			|{~\varPhi (x_k,y_k)}-{~\varPhi (x^*,y^*)}|=\mathcal{O}\left( \frac{1}{\beta_k} \right).
		\end{eqnarray*}
		which completes the proof of this theorem.
	\end{proof}
	\begin{remark}
		The conditions in Assumption \ref{ass} correspond to the continuous case \cite{35}. When the objective functions are convex, although our dynamical system is based on the unaugmented Lagrangian function, applying a proof similar to that in \cite{35} can lead to a same assumption. This ensures that the implicit choice of multipliers inherits the core properties of the continuous dynamical system. Furthermore, both the objective residual and the feasibility violation of our algorithm converge at a rate of $O(1/\beta_k)$, which is faster than the corresponding $O(1/\sqrt{\beta(t)})$ rate in \cite{35}.
	\end{remark}

	Below, we will provide the convergence properties of Algorithm \ref{al_3}.
	\begin{theorem}\label{th3.4}
		Suppose that Assumption \ref{ass4} holds. Let $\{(x_k,y_k,\lambda_k)\}_{k\geq 1}$ be the sequence generated by Algorithm \ref{al_3} and  $(x^*,y^*,\lambda^*)\in\Omega$.  Assume that $ \{\epsilon_{k}\}_{k\geq1} $ is a nonincreasing sequence and $\sum_{k=1}^{+\infty}\alpha_k\beta_{k}\epsilon_{k}<{+\infty}$, then we have the sequence $ \{(x_{k},y_k,\lambda_k)\}_{k\geq1} $ is bounded and the following statements:
		\begin{eqnarray*}
			\begin{aligned}
				&\mathcal{L} ( x_k, y_k,\lambda ^*)-\mathcal{L} ( x^*,y^*,\lambda ^* )=\mathcal{O}\left( \frac{1}{\beta_k} \right),\\
				\| \varPhi ( x_k,y_k)-&\varPhi ( x^*,y^*)\|=\mathcal{O} \left( \frac{1}{\beta_k} \right),\,\,\,\left \| Ax_k+B y_k-b\right \|=\mathcal{O}\left ( \frac{1}{\beta_k}\right ).\\
			\end{aligned}
		\end{eqnarray*}
	\end{theorem}
	
	\begin{proof}
		Since $\hat{\lambda}_{k+1} = {\lambda}_{k+1}$, we get
		\begin{align*}
			&-\delta\alpha_k\beta_k\langle \bar\lambda_{k+1}-\lambda^*,A(Z_{k+1}^\delta-x^*) \rangle
			-\delta\alpha_k\beta_k\langle \hat\lambda_{k+1}-\lambda^*,B(H_{k+1}^\delta-y^*) \rangle\nonumber +\delta\langle\lambda_{k+1}-\lambda_k,\lambda_{k+1}-\lambda^*\rangle\nonumber\\
			=&\delta\alpha_k\beta_k\langle \lambda_{k+1}-\bar\lambda_{k+1},A(Z_{k+1}^\delta-x^*) \rangle
			+\delta\alpha_k\beta_k\langle \lambda_{k+1}-\hat\lambda_{k+1},B(H_{k+1}^\delta-y^*) \rangle\nonumber\\
			=&\delta\alpha_k\beta_k\langle \lambda_{k+1}-\bar\lambda_{k+1},A(Z_{k+1}^\delta-x^*) \rangle.
		\end{align*}
		Further,
		\begin{align*}
			&\delta\alpha_k\beta_k\langle\lambda_{k+1}-\bar\lambda_{k+1},A({Z}^{\delta}_{k+1}-x^*)\rangle \nonumber\\ =&\delta\alpha_k\beta_k\langle\lambda_{k+1}-\bar\lambda_{k+1},A({Z}^{\delta}_{k+1}-x^*)\rangle
			+\delta\alpha_k\beta_k\langle\lambda_{k+1}-\bar\lambda_{k+1},B({H}^{\delta}_{k+1}-y^*)\rangle\nonumber\\
			&+\delta\alpha_k\beta_k\langle\bar\lambda_{k+1}-\lambda_{k+1},B({H}^{\delta}_{k+1}-y^*)\rangle\nonumber\\
			=&\delta\langle{\lambda}_{k+1}-\bar{\lambda}_{k+1}, \lambda_{k+1}-\lambda_{k}\rangle+\delta\alpha_k\beta_k\langle\bar\lambda_{k+1}
			-\lambda_{k+1},B({H}^{\delta}_{k+1}-y^*)\rangle\nonumber\\
			\leq& \frac{\delta}{2}\|\lambda_{k+1}-\bar\lambda_{k+1}\|^2 + \frac{\delta}{2}\|\lambda_{k+1}-\lambda_{k}\|^2+\delta\alpha_k\beta_k\langle\bar\lambda_{k+1}
			-\lambda_{k+1},B({H}^{\delta}_{k+1}-y^*)\rangle.
		\end{align*}
		From the parameter settings of $\lambda_{k+1}$ and $\bar\lambda_{k+1}$, we obtain
		\begin{equation}\label{energyt2.3}
			\begin{aligned}
				&\frac{\delta}{2}\|\lambda_{k+1}-\bar\lambda_{k+1}\|^2 +\delta\alpha_k\beta_k\langle\bar\lambda_{k+1}-\lambda_{k+1},B({H}^{\delta}_{k+1}-y^*)\rangle \\
				=&\frac{\delta\alpha_k^2\beta_k^2}{2}\|B({H}^{\delta}_{k+1}-{H}^{\delta}_{k})\|^2-\delta\alpha_k^2\beta_k^2\langle B({H}^{\delta}_{k+1}-{H}^{\delta}_{k}),
				B({H}^{\delta}_{k+1}-y^*)\rangle\\
				=&-\frac{\delta\alpha_k^2\beta_k^2}{2}\|B({H}^{\delta}_{k+1}-y^*)\|^2 + \frac{\delta\alpha_k^2\beta_k^2}{2}\|B({H}^{\delta}_{k}-y^*)\|^2.
			\end{aligned}
		\end{equation}
		where the second equality follows from (\ref{ds}). Substituting the above expression into (\ref{energybds}), we get
		\begin{equation}\label{term}
			\begin{aligned}
				&\mathcal{E}_{k+1}+\frac{\delta\alpha_{k+1}^2\beta_{k+1}^2}{2}\|B({H}^{\delta}_{k+1}-y^*)\|^2 - (\mathcal{E}_k  + \frac{\delta\alpha_k^2\beta_{k}^2}{2}\|B({H}^{\delta}_{k}-y^*)\|^2)\\
				\leq&({\delta}^2{\beta}_{k+1}-{\delta}^2{\beta}_{k}-{\delta}\alpha_k{\beta}_{k})(\mathcal{L}(x_{k+1},y_{k+1},\lambda^*)-\mathcal{L}(x^*,y^*,\lambda^*))\\
				&+\frac{{\delta}\alpha_k{\beta}_k{\epsilon}_k}{2}(\|x^*\|^2+\|y^*\|^2)\\
				&+(\frac{{\delta}^2{\beta}_{k+1}{\epsilon}_{k+1}}{2}
				-\frac{{\delta}^2{\beta}_k{\epsilon}_k}{2}-\frac{{\delta}\alpha_k{\beta}_k{\epsilon}_k}{2})(\|x_{k+1}\|^2+\|y_{k+1}\|^2)\\ &-\frac{1}{2}\|Z_{k+1}^\delta-Z_{k}^\delta\|^2-\frac{1}{2}\|H_{k+1}^\delta-H_{k}^\delta\|^2-\frac{{\delta}^2{\beta}_k{\mu}_f}{2}\|x_{k+1}-x_k\|^2\\
				&-\frac{{\delta}^2{\beta}_k{\mu}_g}{2}\|y_{k+1}-y_k\|^2-\frac{\delta\alpha_k\beta_k\mu_f}{2}\|Z_{k+1}^\delta-x^*\|^2\\
				&\underbrace{-\frac{\delta\alpha_k\beta_k\mu_g}{2}\|H_{k+1}^\delta-y^*\|^2+\frac{\delta(\alpha_{k+1}^2\beta_{k+1}^2-\alpha_{k}^2\beta_k^2)}{2}\|B(H_{k+1}^\delta-y^*)\|^2}_{\mathbb{I}}\\
				&-\frac{\delta\alpha_k\beta_k\mu_f}{2}\|x_{k+1}-Z_{k+1}^\delta\|^2-\frac{\delta\alpha_k\beta_k\mu_g}{2}\|y_{k+1}-H_{k+1}^\delta\|^2\\
				&-(\delta\gamma-1)(\frac{\delta}{a_k}+\frac{1}{2})\|x_{k+1}-x_k\|^2-(\delta\gamma-1)(\frac{\delta}{a_k}+\frac{1}{2})\|y_{k+1}-y_k\|^2,
			\end{aligned}
		\end{equation}
		where $\mu_f \geq 0$ and $\mu_g \geq 0$.
		Let us focus on the term $\mathbb{I}$. For $\alpha_{k+1}\beta_{k+1} \leq \alpha_k\beta_k$, it is evident that $\mathbb{I}\leq0$. In addition, for $\alpha_{k+1}\beta_{k+1} > \alpha_k\beta_k$, in view of Assumption \ref{ass4}, we get
		\begin{align*}
			&\mathbb{I}\leq-\frac{\delta\alpha_k\beta_k\mu_g}{2}\|H_{k+1}^\delta-y^*\|^2+\frac{\delta\|B\|^2(\alpha_{k+1}^2\beta_{k+1}^2
				-\alpha_{k}^2\beta_k^2)}{2}\|H_{k+1}^\delta-y^*\|^2\leq0.
		\end{align*}
		Hence, combining (\ref{term}) with Assumption \ref{ass4} gives
		\begin{equation}\label{energyt2.5}
			\begin{aligned}
				&\mathcal{E}_{k+1}+\frac{\delta\alpha_{k+1}^2\beta_{k+1}^2}{2}\|B({H}^{\delta}_{k+1}-y^*)\|^2 - (\mathcal{E}_k  + \frac{\delta\alpha_{k}^2{\beta_{k}}^2}{2}\|B({H}^{\delta}_{k}-y^*)\|^2)\\
				\leq&\frac{{\delta}\alpha_k{\beta}_k{\epsilon}_k}{2}(\|x^*\|^2+\|y^*\|^2).
			\end{aligned}
		\end{equation}
		Along with $\sum_{k=1}^{+\infty}\alpha_k\beta_{k}\epsilon_{k}<{+\infty}$, it implies that
		\begin{align*}
			\mathcal{E}_{k+1}&+\frac{\delta\alpha_{k+1}^2\beta_{k+1}^2}{2}\|B({H}^{\delta}_{k+1}-y^*)\|^2 \leq \mathcal{E}_1  + \frac{\delta\alpha_{1}^2\beta_{1}^2}{2}\|B({H}^{\delta}_{1}-y^*)\|^2+C_3,
		\end{align*}
		where $C_3$ is a constant. Now, using the definition of $\mathcal{E}_{k}$, we get
		\begin{align*}
			\mathcal{L} ( x_k, y_k,\lambda ^*)-\mathcal{L} ( x^*,y^*,\lambda ^* )=\mathcal{O}\left( \frac{1}{\beta_k} \right).
		\end{align*}
		The remaining proof is similar to Theorem \ref{th3.1}, thus we omit it here.
	\end{proof}
	\begin{corollary}\label{co3.2t}
		Let $\beta_k=k$, $\alpha_k=\frac{1}{k}$ and $\epsilon_k=\frac{1}{k^3}$. Suppose that $\mu_g=0$, $\delta\leq 1$ and $\delta\gamma-1\geq 0$ holds. Let $\{(x_k,y_k,\lambda_k)\}_{k\geq 1}$ be the sequence generated by Algorithm \ref{al_3} and $(x^*,y^*,\lambda^*)\in\Omega$. Then we have the sequence $ \{(x_{k},y_k,\lambda_k)\}_{k\geq1} $ is bounded and the following statements:
		\begin{eqnarray*}
			\begin{aligned}
				&\mathcal{L} ( x_k, y_k,\lambda ^*)-\mathcal{L} ( x^*,y^*,\lambda ^* )=\mathcal{O}\left( \frac{1}{k} \right),\\
				\| f ( x_k,y_k)-& f ( x^*,y^*)\|=\mathcal{O} \left( \frac{1}{k} \right),\,\,\,\left \| Ax_k+By_k-b\right \|=\mathcal{O}\left ( \frac{1}{k}\right ).\\
			\end{aligned}
		\end{eqnarray*}
	\end{corollary}
	\begin{corollary}\label{co3.2}
		Let $\beta_k=\frac{\mu_gk^2}{3\|B\|^2}$, $\alpha_k=\frac{1}{k}$ and $\epsilon_k=\frac{1}{\alpha_k\beta_kk^3}$. Suppose that $\mu_g>0$, $\delta\leq\frac{1}{3}$ and $\delta\gamma-1\geq 0$ holds. Let $\{(x_k,y_k,\lambda_k)\}_{k\geq 1}$ be the sequence generated by Algorithm \ref{al_3} and  $(x^*,y^*,\lambda^*)\in\Omega$. Then we have the sequence $ \{(x_{k},y_k,\lambda_k)\}_{k\geq1} $ is bounded and the following statements:
		\begin{eqnarray*}
			\begin{aligned}
				&\mathcal{L} ( x_k, y_k,\lambda ^*)-\mathcal{L} ( x^*,y^*,\lambda ^* )=\mathcal{O}\left( \frac{1}{k^2} \right),\\
				\| f ( x_k,y_k)-& f ( x^*,y^*)\|=\mathcal{O} \left( \frac{1}{k^2} \right),\,\,\,\left \| Ax_k+By_k-b\right \|=\mathcal{O}\left ( \frac{1}{k^2}\right ).\\
			\end{aligned}
		\end{eqnarray*}
	\end{corollary}
	\begin{remark}
		Corollaries \ref{co3.2t} and \ref{co3.2} provide two examples showing that Algorithm \ref{al_3} achieves non-ergodic convergence rates of $O(1/k)$ for convex objectives and $O(1/k^2)$ for partially strongly convex objectives (where $\mu_g > 0$). These rates match those in \cite{30}, but we relax the equality condition between the step size and the time scale by replacing the equality in scheme (3.2) of \cite{30} with the inequality $\delta\beta_{k+1} \leq \delta\beta_k + \alpha_k\beta_k$.
	\end{remark}
	\begin{remark}
		For partially strongly convex objectives, Algorithm \ref{al_3} can only achieve its best possible convergence rate of $O(1/k)$ when the step size $\alpha_k = 1$. However, adopting a varying step size $\{\alpha_k\}_{k\geq1}$, we can achieve a faster convergence rate $O(1/k^2)$.
		
	\end{remark}
	\begin{remark}
		In fact, besides the results from Theorems \ref{th3.1} and \ref{th3.4}, we can derive the convergence rate of the trajectory when $f$ and $g$ are strongly convex.
		According to $\mathcal{L} ( x_k, y_k,\lambda ^*)-\mathcal{L} ( x^*,y^*,\lambda ^* )=\mathcal{O}\left( \frac{1}{\beta_k} \right)$
		and the following inequality
		\begin{equation*} \frac{{\mu}_f}{2}\|x_k-x^*\|^2+\frac{{\mu}_g}{2}\|y_k-y^*\|^2\leq {\mathcal{L}(x_k,y_k,{\lambda^*})-\mathcal{L}(x^*,y^*,{\lambda^*})-\langle u,x_k-x^*\rangle-\langle v, y_k-y^*\rangle},
		\end{equation*}
		where $u\in\partial_x \mathcal{L}(x^*,y^*,{\lambda^*})$ and $v\in\partial_y \mathcal{L}(x^*,y^*,{\lambda^*})$, %Indeed, by taking $(x_1,y_1)=(x_k,y_k)$ and $(x_2,y_2)=(x^*,y^*)$
		we immediately deduce that
		\begin{equation*}
			{\mu}_f\|x_k-x^*\|^2+\mu_g\|y_k-y^*\|^2=\mathcal{O} \left( \frac{1}{\beta_k} \right).
		\end{equation*}
		Moreover, from (\ref{energybds}) together with Theorems \ref{th3.1} and \ref{th3.4}, we obtain
		\begin{align*}
			\mathcal{E}_{k+1}-\mathcal{E}_k \leq -\frac{{\delta}^2{\beta}_k{\mu}_f}{2}\|x_{k+1}-x_k\|^2-\frac{{\delta}^2{\beta}_k{\mu}_g}{2}\|y_{k+1}-y_k\|^2.
		\end{align*}
		Summing the above inequality from $1$ to $K$, we get
		\begin{align*} \mathcal{E}_{K}+\frac{{\delta}^2{\mu}_f}{2}\sum_{k=1}^{K-1}{\beta}_k\|x_{k+1}-x_k\|^2+\frac{{\delta}^2{\mu}_g}{2}\sum_{k=1}^{K-1}{\beta}_k\|y_{k+1}-y_k\|^2<+\infty.
		\end{align*}
		It follows that
		\begin{align*}
			\underset{k\rightarrow+\infty}{\lim} \mu_f\beta_k\|x_{k+1}-x_k\|^2=0,\,\,\,\underset{k\rightarrow+\infty}{\lim} \mu_g\beta_k\|y_{k+1}-y_k\|^2=0.
		\end{align*}
		Hence, we obtain the desired conclusion:
		\begin{eqnarray*}
			\begin{aligned}
				\| x_k-x^*\|^2=\mathcal{O} \left( \frac{1}{\beta_k} \right),\,\,\,\left \| x_{k+1}-x_k\right \|^2=o\left ( \frac{1}{\beta_k}\right ),\;for\;\mu_f>0;\\
			\end{aligned}
		\end{eqnarray*}
		\begin{eqnarray*}
			\begin{aligned}
				\| y_k-y^*\|^2=\mathcal{O} \left( \frac{1}{\beta_k} \right),\,\,\,\left \| y_{k+1}-y_k\right \|^2=o\left ( \frac{1}{\beta_k}\right ),\;for\;\mu_g>0.\\
			\end{aligned}
		\end{eqnarray*}
	\end{remark}
	
	\section{Strong convergence of sequences to the minimal norm solution}
	
	In this section, we investigate the strong convergence properties of the sequences generated by the algorithms mentioned before. Let $(\bar x^*, \bar y^*)$ be the minimal norm of the solution set $\mathbb S$, i.e.,$(\bar x^*, \bar y^*)=proj_{\mathbb S} 0$, where $proj$ denotes the projection operator. Then, there exists an optimal dual solution $\bar\lambda^*\in Z $ for problem (\ref{max min}) such that $(\bar x^*,\bar y^*,\bar \lambda^*)\in \Omega$. For any $\epsilon>0$, we set
	\begin{align*}
		(x_\epsilon,y_\epsilon):=\underset{x\in \mathcal{X}, y\in \mathcal{Y}}{\arg\min}\,\mathcal{L}^\epsilon(x,y,{\bar\lambda^*})
	\end{align*}
	with
	\begin{align*}
		\mathcal{L}^\epsilon(x,y,{\bar\lambda^*}) = \mathcal{L}(x,y,{\bar\lambda^*})+\frac{{\epsilon}}{2}\left\|x\right\|^2+\frac{{\epsilon}}{2}\left\|y\right\|^2.
	\end{align*}
	Clearly, from the optimality condition, we get
	\begin{equation*}
		\begin{cases}
			0 \in \partial_x \mathcal{L}^\epsilon(x_\epsilon,y_\epsilon,{\bar\lambda^*})=\partial_x \mathcal{L}(x_\epsilon,y_\epsilon,\bar\lambda^*)+\epsilon x_\epsilon,\\
			0 \in \partial_y \mathcal{L}^\epsilon(x_\epsilon,y_\epsilon,{\bar\lambda^*})=\partial_y \mathcal{L}(x_\epsilon,y_\epsilon,\bar\lambda^*)+\epsilon y_\epsilon.
		\end{cases}
	\end{equation*}
	Due to the classical properties of the Tikhonov regularization, we get
	\begin{align*}
		\|x_\epsilon\| \leq \|\bar x^*\|,\,\forall \epsilon>0,\qquad\lim_{k\rightarrow{+\infty}}\|x_\epsilon-\bar x^*\|=0,
	\end{align*}
	and
	\begin{align}\label{4.5}
		\|y_\epsilon\| \leq \|\bar y^*\|,\,\forall \epsilon>0,\qquad\lim_{k\rightarrow{+\infty}}\|y_\epsilon-\bar y^*\|=0.
	\end{align}

	Now, we recall the following lemma and establish two propositions that will facilitate the subsequent proofs.
	\begin{lemma}\label{A.2}\textup{\cite[Lemma IV.3.2]{7}}
		Suppose that $(\varphi_k)^\infty_{k=1}$ is an infinite sequence of real numbers such that $\sum_{k=1}^{+\infty} {\varphi_k}=s$ exists and is finite, where $k\geq 1$. And let $(\psi_k)^\infty_{k=1}$ be a nondecreasing function in $(0,+\infty)$ such that $\underset{k\rightarrow+\infty}{\lim} \psi_k =+\infty $. Then
		\begin{align*}
			\underset{k\rightarrow+\infty}{\lim} \frac{1}{\psi_k}\sum_{i=1}^{k}{\psi_i\varphi_i}=0.
		\end{align*}
	\end{lemma}
	\begin{proposition}\label{pr1}
		Suppose that Assumption \ref{ass} holds and $ \{\epsilon_{k}\}_{k\geq1} $ is a nonincreasing sequence satisfying $\sum_{k=1}^{+\infty}\alpha_k\epsilon_{k}<{+\infty}$. Let $ \{(x_{k},y_{k},\lambda_{k})\}_{k\geq1} $ be the sequence generated by Algorithm \ref{al_1} (and likewise for Assumptions \ref{ass4} corresponding to Algorithms \ref{al_3}). For every $k\geq1$ and $(x^*,y^*)\in \mathbb S $, it holds
		\begin{align*}
			\underset{k\rightarrow+\infty}{\lim} \mathcal{L}(x_k,y_k,\lambda^*)-\mathcal{L}(x^*,y^*,\lambda^*)=0.
		\end{align*}
	\end{proposition}
	\begin{proof}
		Using a similar arguments as (\ref{energyt1.3}) and (\ref{energyt2.5}), we obtain
		\begin{align*}
			\frac{\mathcal{E}_{k}}{\delta^2\beta_k} \leq \frac{\mathcal{E}_{1}}{\delta^2\beta_k}+\frac{\|x^*\|^2+\|y^*\|^2}{2\delta\beta_k}\sum_{i=1}^{k-1}\alpha_i\beta_{i}\epsilon_{i}
			\leq \frac{\mathcal{E}_{1}}{\delta^2\beta_k}+\frac{\|x^*\|^2+\|y^*\|^2}{2\delta\beta_{k-1}}\sum_{i=1}^{k-1}\alpha_i\beta_{i}\epsilon_{i},
		\end{align*}
		and
		\begin{align*}
			\frac{\mathcal{E}_{k}}{\delta^2\beta_k}&\leq \frac{\mathcal{E}_{1}}{\delta^2\beta_k}
			+\frac{\beta_{1}^2}{2\delta\beta_k}\|B({H}^{\delta}_{1}-y^*)\|^2+
			\frac{\|x^*\|^2+\|y^*\|^2}{2\delta\beta_k}\sum_{i=1}^{k-1}\alpha_i\beta_{i}\epsilon_{i}\\
			&\leq \frac{\mathcal{E}_{1}}{\delta^2\beta_k}
			+\frac{\beta_{1}^2}{2\delta\beta_k}\|B({H}^{\delta}_{1}-y^*)\|^2+
			\frac{\|x^*\|^2+\|y^*\|^2}{2\delta\beta_{k-1}}\sum_{i=1}^{k-1}\alpha_i\beta_{i}\epsilon_{i}.
		\end{align*}
		Applying Lemma \ref{A.2} with $\varphi_k=\alpha_k\epsilon_k$ and $\psi_k=\beta_k$, we have
		\begin{align*}
			\underset{k\rightarrow+\infty}{\lim} \frac{1}{\beta_k}\sum_{i=1}^{k}\alpha_i\beta_i\epsilon_i=0.
		\end{align*}
		As $\lim_{k\rightarrow{+\infty}}\beta_{k}={+\infty}$, we can deduce that
		\begin{align*}
			\underset{k\rightarrow+\infty}{\lim} \frac{\mathcal{E}_{1}}{\delta^2\beta_k}=0 \qquad and\qquad \underset{k\rightarrow+\infty}{\lim} \frac{\mathcal{E}_{1}}{\delta^2\beta_k}
			+\frac{\beta_{1}^2}{2\delta\beta_k}\|B({H}^{\delta}_{1}-y^*)\|^2=0.
		\end{align*}
		These together with $(\ref{energy})$ complete the proof.
	\end{proof}
	\begin{proposition}\label{pr2}
		Let $( \bar{x}^*,\bar{y}^* )={proj}_\mathbb{S} 0$ and $\{( x_k,y_k,\lambda_k )\}_{k\geq 1}$ be the sequence generated by the discretization scheme $(\ref{disc})$. Then,
		\begin{align*}
			\mathcal{L}^{\epsilon_k}( x_k,y_k,{\bar\lambda^*})-\mathcal{L}^{\epsilon_k}( \bar{x}^*,\bar{y}^*,{\bar\lambda^*} )\geq &\frac{\epsilon_k}{2} \|(x_k, y_k)-(x_{\epsilon_k},y_{\epsilon_k})\|^2\\
			&+\frac{\epsilon_k}{2}( \| x_{\epsilon_k}\|^2-\| \bar{x}^*\|^2+ \| y_{\epsilon _k} \|^2- \| \bar{y}^* \|^2).
		\end{align*}
	\end{proposition}
	\begin{proof}
		Since the proof follows a similar argument to Proposition 5.1 in \cite{35}, we refer the reader to \cite{35} for details and omit it here.
	\end{proof}

	In what following, we prove the strong convergence of the sequences generated by our proposed algorithms.
	\begin{theorem}\label{th4.1}
		Suppose that Assumption \ref{ass} holds. Let $\{(x_k,y_k,\lambda_k)\}_{k\geq 1}$ be the sequence generated by Algorithm \ref{al_1} and  $(x^*,y^*,\lambda^*)\in\Omega$.  Assume that $\{\epsilon_{k}\}_{k\geq1} $ is a nonincreasing sequence satisfying $\sum_{k=1}^{+\infty}\alpha_k\epsilon_{k}<{+\infty}$ and $\lim_{k\rightarrow{+\infty}}\beta_{k}\epsilon_{k}={+\infty}$, then
		\begin{eqnarray*}\label{q2.1}
			\underset{k\rightarrow{+\infty}}{\lim\inf}\|(x_k,y_k)- (\bar x^*,\bar y^*)\|=0.
		\end{eqnarray*}
		Further, if there exists an integer $K\geq 1$ such that the sequence $\{(x_k,y_k)\}_{k\geq K}$ stays in either the open ball ${B}\left( 0,\sqrt{ \| \bar{x}^* \|^2+\| \bar{y}^*\|^2}\right)$ or its complement, then
		\begin{eqnarray*}\label{q2.2}
			\lim_{k\rightarrow{+\infty}}\|(x_k,y_k)-(\bar x^*,\bar y^*)\|=0.
		\end{eqnarray*}
	\end{theorem}
	\begin{proof}
		To prove the conclusion, we divided the proof into three cases based on the sign of $\| x_k\|^2+\| y_k \|^2$ and $\| \bar{x}^*\|^2+\| \bar{y}^*\|^2$.

		$\mathbf{ Case~ I}$: Suppose that there exists $K\geq 1$ such that the sequence $\{(x_k,y_k)\}_{k\geq K}$ stays in the complement of the ball ${B}\left( 0,\sqrt{ \| \bar{x}^* \|^2+\| \bar{y}^*\|^2}\right)$. This is, $\| x_k\|^2+\| y_k \|^2\geq \| \bar{x}^*\|^2+\| \bar{y}^*\|^2$. Now we define the energy function:
		\begin{align*}
			\widetilde{\mathcal{E}}_k :=& {\delta}^2{\beta}_k(\mathcal{L}(x_k,y_k,\bar\lambda^*)-\mathcal{L}(\bar x^*,\bar y^*,\bar\lambda^*)+\frac{{\epsilon}_k}{2}\left\|x_k\right\|^2
			+\frac{{\epsilon}_k}{2}\left\|y_k\right\|^2)+ \frac{1}{2}\|{{Z}^{\delta}_{k}}-\bar x^*\|^2\nonumber \\
			&+\frac{1}{2}\|{{H}^{\delta}_{k}}-\bar y^*\|^2 + \frac{\delta\gamma-1}{2}\|x_k-\bar x^*\|^2+\frac{\delta\gamma-1}{2}\|y_k-\bar y^*\|^2+ \frac{\delta}{2}\|\lambda_k-\bar\lambda^*\|^2.
		\end{align*}
		and
		\begin{equation}\label{energy3}
			\begin{aligned}
				\hat{\mathcal{E}}_k :=& \frac{\widetilde{\mathcal{E}}_k}{\delta^2\beta_k}-\frac{{\epsilon}_k}{2}\left\|\bar x^*\right\|^2
				-\frac{{\epsilon}_k}{2}\left\|\bar y^*\right\|^2 \\
				=&\mathcal{L}^{\epsilon_k}(x_k,y_k,{\bar\lambda^*}) - \mathcal{L}^{\epsilon_k}(\bar x^*,\bar y^*,{\bar\lambda^*}) + \frac{1}{2\delta^2\beta_k}\|{{Z}^{\delta}_{k}}-\bar x^*\|^2+\frac{1}{2\delta^2\beta_k}\|{{H}^{\delta}_{k}}-\bar y^*\|^2 \\
				&+\frac{\delta\gamma-1}{2\delta^2\beta_k}\|x_k-\bar x^*\|^2
				+ \frac{\delta\gamma-1}{2\delta^2\beta_k}\|y_k-\bar y^*\|^2 + \frac{1}{2\delta\beta_k}\|\lambda_k-\bar \lambda^*\|^2.
			\end{aligned}
		\end{equation}
		It is obvious that
		{\small
			\begin{equation}\label{q2.3}
				\begin{aligned}
					&\delta^2\beta_{k+1} \hat{\mathcal{E}}_{k+1}- \delta^2\beta_{k} \hat{\mathcal{E}}_{k}\\
					=&\delta^2\beta_{k+1}(\frac{\widetilde{\mathcal{E}}_{k+1}}{\delta^2\beta_{k+1}}-\frac{{\epsilon}_{k+1}}{2}\left\|\bar x^*\right\|^2
					-\frac{{\epsilon}_{k+1}}{2}\left\|\bar y^*\right\|^2)-\delta^2\beta_{k}(\frac{\widetilde{\mathcal{E}}_{k}}{\delta^2\beta_{k}}-\frac{{\epsilon}_{k}}{2}\left\|\bar x^*\right\|^2
					-\frac{{\epsilon}_{k}}{2}\left\|\bar y^*\right\|^2) \\
					=&\widetilde{\mathcal{E}}_{k+1} -\widetilde{\mathcal{E}}_{k}- \frac{\delta^2}{2}\beta_{k+1}\epsilon_{k+1}(\|\bar x^*\|^2+\|\bar y^*\|^2)+\frac{\delta^2}{2}\beta_k\epsilon_k(\|\bar x^*\|^2+\|\bar y^*\|^2).
				\end{aligned}
		\end{equation}}
		According to (\ref{energyt1.1}), we have
		\begin{align*}
			&\delta^2\beta_{k+1}\hat{\mathcal{E}}_{k+1}-\delta^2\beta_{k} \hat{\mathcal{E}}_{k} \nonumber\\
			\leq&({\delta}^2{\beta}_{k+1}-{\delta}^2{\beta}_{k}-{\delta}\alpha_k{\beta}_{k})(\mathcal{L}(x_{k+1},y_{k+1},\bar\lambda^*)
			-\mathcal{L}(x^*,y^*,\bar\lambda^*))\nonumber\\
			&+(\frac{{\delta}^2{\beta}_{k+1}{\epsilon}_{k+1}}{2}
			-\frac{{\delta}^2{\beta}_k{\epsilon}_k}{2}-\frac{{\delta}\alpha_k{\beta}_k{\epsilon}_k}{2})(\|x_{k+1}\|^2+\|y_{k+1}\|^2-\|\bar x^*\|^2-\|\bar y^*\|^2)\nonumber\\ &-\frac{1}{2}\|Z_{k+1}^\delta-Z_{k}^\delta\|^2-\frac{1}{2}\|H_{k+1}^\delta-H_{k}^\delta\|^2-\frac{\delta}{2}\|{\lambda}_{k+1}-{\lambda}_{k}\|^2\nonumber\\
			&-\frac{{\delta}^2{\beta}_k{\mu}_f}{2}\|x_{k+1}-x_k\|^2-\frac{{\delta}^2{\beta}_k{\mu}_g}{2}\|y_{k+1}-y_k\|^2\nonumber\\
			&-\frac{\delta\alpha_k\beta_k\mu_f}{2}\|Z_{k+1}^\delta-\bar x^*\|^2-\frac{\delta\alpha_k\beta_k\mu_g}{2}\|H_{k+1}^\delta-\bar y^*\|^2\nonumber\\
			&-\frac{\delta\alpha_k\beta_k\mu_f}{2}\|x_{k+1}-Z_{k+1}^\delta\|^2-\frac{\delta\alpha_k\beta_k\mu_g}{2}\|y_{k+1}-H_{k+1}^\delta\|^2\nonumber\\
			&-(\delta\gamma-1)(\frac{\delta}{a_k}+\frac{1}{2})\|x_{k+1}-x_k\|^2-(\delta\gamma-1)(\frac{\delta}{a_k}+\frac{1}{2})\|y_{k+1}-y_k\|^2.
		\end{align*}
		It follows from Assumption $\ref{ass}$ that
		\begin{align*}
			\delta^2\beta_{k+1} \hat{\mathcal{E}}_{k+1}\leq \delta^2\beta_{k} \hat{\mathcal{E}}_{k}.
		\end{align*}
		Thus, we obtain
		\begin{align*}
			\delta^2\beta_{k} \hat{\mathcal{E}_{k}}\leq \delta^2\beta_{K} \hat{\mathcal{E}}_{K},\,\,\forall k\geq K.
		\end{align*}
		Due to the definition of $\hat{\mathcal{E}}_{k}$, we obtain
		\begin{align*}
			\mathcal{L}^{\epsilon_k}(x_k,y_k,{\bar\lambda^*})- \mathcal{L}^{\epsilon_k}(\bar x^*,\bar y^*,{\bar\lambda^*}) \leq \frac{{\beta}_{K}}{\beta_k}\hat{\mathcal{E}}_{K}.
		\end{align*}
		Applying Proposition \ref{pr2}, we get
		\begin{align*}
			\|(x_k,y_k)- (x_{\epsilon_k},y_{\epsilon_k})\|^2 \leq \frac{2{\beta}_{K}}{\beta_k\epsilon_k}\hat{\mathcal{E}}_{K} +
			\|\bar x^*\|^2-\|x_{\epsilon_k}\|^2+\|\bar y^*\|^2-\|y_{\epsilon_k}\|^2.
		\end{align*}
		This together with $\lim_{k\rightarrow{+\infty}}\beta_{k}\epsilon_{k}={+\infty}$ and (\ref{4.5}) implies
		\begin{eqnarray*}
			\lim_{k\rightarrow{+\infty}}\|(x_k,y_k)-(\bar x^*,\bar y^*)\|=0.
		\end{eqnarray*}

		$\mathbf{ Case~ II}$: Suppose that there exists $K\geq 1$ such that the sequence $\{(x_k,y_k)\}_{k\geq K}$ stays in the ball ${B}\left( 0,\sqrt{ \| \bar{x}^* \|^2+\| \bar{y}^*\|^2}\right)$ such that $ \| x_k\|^2+\| y_k \|^2 < \| \bar{x}^*\|^2+\| \bar{y}^*\|^2 $. Let $ (\bar x,\bar y)\in \mathcal{X\times Y}$ be a weak sequential cluster of ${\{(x_k,y_k)\}}_{k\geq 1}$. Then, there exists a subsequence $(x_{k_j},y_{k_j})_{j\geq1}$ such that ${k_j\rightarrow{+\infty}} $ and $(x_{k_j},x_{k_j})$ converges weakly to ${(\bar x,\bar y)} $ as $ j\rightarrow +\infty$. Because $\mathcal{L}(x,y,\bar \lambda^*)$ is convex and lower semi-continuous, we have
		\begin{eqnarray*}
			\mathcal{L}(\bar x,\bar y,\bar \lambda^*)\leq \underset{j\rightarrow{+\infty}}{\lim\inf}\mathcal{L}(x_{k_j},y_{k_j},\bar \lambda^*).
		\end{eqnarray*}
		Using Proposition \ref{pr1}, we get
		\begin{eqnarray*}
			\mathcal{L}(\bar x,\bar y,\bar \lambda^*)\leq\underset{j\rightarrow{+\infty}}{\lim\inf}\mathcal{L}(x_{k_j},y_{k_j},\bar \lambda^*) = \mathcal{L}(\bar x^*,\bar y^*,\bar \lambda^*).
		\end{eqnarray*}
		This, combined with $\mathcal{L}(\bar x^*,\bar y^*,\bar \lambda^*)=\underset{x\in \mathcal X,y\in \mathcal Y}{\min}\mathcal{L}(x,y,\bar \lambda^*)$, result in
		\begin{eqnarray*}
			\mathcal{L}(\bar x,\bar y,\bar \lambda^*)\leq \mathcal{L}(\bar x^*,\bar y^*,\bar \lambda^*)=\underset{x\in \mathcal X,y\in \mathcal Y}{\min}\mathcal{L}(x,y,\bar \lambda^*)\leq\mathcal{L}(\bar x,\bar y,\bar \lambda^*).
		\end{eqnarray*}
		This means $(\bar x,\bar y)\in \mathcal D(\bar\lambda^*)$, then we have $(\bar x,\bar y)\in \mathbb S $. Using the weak lower semi-continuity of $ \|\cdot\|$, we get
		\begin{eqnarray*}
			\| ( \bar{x},\bar{y} ) \|\leq \liminf_{j\rightarrow+\infty} \|( x_{k_j},y_{k_j})\|\leq\limsup_{j\rightarrow+\infty}{ \| ( x_{k_j},y_{k_j}) \|}\leq \| (\bar{x}^*,\bar{ y}^*)\|.
		\end{eqnarray*}
		Combined with $(\bar x^*,\bar y^*)=proj_{\mathbb S} 0$, we have $(\bar x,\bar y)=(\bar x^*,\bar y^*)$. Moreover, the sequence $\{(x_{k},y_{k})\}_{k\geq1}$ has a unique weak cluster point $(\bar x^*,\bar y^*)$. Therefore, this shows that the sequence converges weakly to $(\bar x^*,\bar y^*)$.
		Thus,
		\begin{eqnarray*}
			\| ( \bar{x}^*,\bar{y}^* ) \|\leq \liminf_{j\rightarrow+\infty} \|( x_{k},y_{k})\|\leq\limsup_{j\rightarrow+\infty}{ \| ( x_{k},y_{k}) \|}\leq \| (\bar{x}^*,\bar{ y}^*)\|,
		\end{eqnarray*}
		which gives
		$$
		\lim_{k\rightarrow+\infty} \| ( x_k,y_k)\|= \| ( \bar{x}^*,\bar{y}^*)\|.
		$$
		This together with the fact that the sequence $\{(x_{k},y_{k})\}_{k\geq1}$ converges weakly to $(\bar x,\bar y)$, we obtain that the convergence is strong, that is
		\begin{eqnarray*}
			\lim_{k\rightarrow+\infty} \| ( x_k,y_k )-( \bar{x}^*,\bar{y}^* )\|=0.
		\end{eqnarray*}

		$\mathbf{ Case~ III}$: Suppose that for any integer ${K}\geq 1$, there exists $k\geq K$ such that $ \| x_k\|^2+\| y_k \|^2 < \| \bar{x}^*\|^2+\| \bar{y}^*\|^2 $ and there exists $s\geq K$ such that $ \| x_s\|^2+\| y_s \|^2 \geq \| \bar{x}^*\|^2+\| \bar{y}^*\|^2 $. By the continuity, it follows that here exists a subsequence $\{(x_{k_j},y_{k_j})\}_{j\geq1}$ of $\{(x_k,y_k)\}_{k\geq1}$ such that $ \| x_{k_j}\|^2+\| y_{k_j} \|^2 = \| \bar{x}^*\|^2+\| \bar{y}^*\|^2 $. Thus
		\begin{eqnarray}\label{xz}
			\lim_{j\rightarrow+\infty} \| ( x_{k_j},y_{k_j} )\|=\|( \bar{x}^*,\bar{y}^* )\|.
		\end{eqnarray}
		Let $ (\hat x,\hat y)\in \mathcal{X\times Y}$ be a weak sequential cluster of $\{(x_{k_j},y_{k_j})\}_{j\geq1}$. Using similar arguments in case II, we obtain that $(\hat x,\hat y)=(\bar{x}^*,\bar{y}^*)$ and $\{(x_{k_j},y_{k_j})\}_{j\geq1}$ converges weakly to $(\bar{x}^*,\bar{y}^*)$ as $j\rightarrow+\infty$. This together with \eqref{xz} yields
		\begin{eqnarray*}
			\lim_{j\rightarrow+\infty} \| ( x_{k_j},y_{k_j} )-( \bar{x}^*,\bar{y}^* )\|=0.
		\end{eqnarray*}
		As a result,
		\begin{eqnarray*}
			\underset{j\rightarrow+\infty}{\lim\inf} \| ( x_{k},y_{k} )-( \bar{x}^*,\bar{y}^* )\|=0.
		\end{eqnarray*}
		This completes the proof.
	\end{proof}

	Next, we aim to analyze the strong convergence of the sequence generated by Algorithm \ref{al_3}.
	\begin{theorem}\label{th4.3}
		Suppose that Assumption \ref{ass4} holds. Let $\{(x_k,y_k,\lambda_k)\}_{k\geq 1}$ be the sequence generated by Algorithm \ref{al_3} and  $(x^*,y^*,\lambda^*)\in\Omega$.  Assume that $ \{\epsilon_{k}\}_{k\geq1} $ is a nonincreasing sequence satisfying $\sum_{k=1}^{+\infty}\alpha_k\epsilon_{k}<{+\infty}$ and $\lim_{k\rightarrow{+\infty}}\beta_{k}\epsilon_{k}={+\infty}$, then
		\begin{eqnarray*}
			\underset{k\rightarrow{+\infty}}{\lim\inf}\|(x_k,y_k)- (\bar x^*,\bar y^*)\|=0.
		\end{eqnarray*}
		Further, if there exists an integer $K\geq1$ such that the sequence $\{x_k\}_{k\geq K}$ stays in either the open ball ${B}\left( 0,\sqrt{ \| \bar{x}^* \|^2+\| \bar{y}^*\|^2}\right)$ or its complement, then
		\begin{eqnarray*}
			\lim_{k\rightarrow{+\infty}}\|(x_k,y_k)-(\bar x^*,\bar y^*)\|=0.
		\end{eqnarray*}
	\end{theorem}
	\begin{proof}
		Following the same structure as the proof of Theorem \ref{th4.1}, we divide our proof into three cases. Now, we just focus on proving Case I where $\| x_k\|^2+\| y_k \|^2\geq \| \bar{x}^*\|^2+\| \bar{y}^*\|^2$, for all $k\geq K$. Our proof still uses this energy function $(\ref{energy3})$. By the proof of Theorem \ref{th3.4} and (\ref{q2.3}), we have
		\begin{align*}
			&\delta^2\beta_{k+1}\hat{\mathcal{E}}_{k+1}+\frac{\delta\alpha_{k+1}^2{\beta}_{k+1}^2}{2}\|B({H}^{\delta}_{k+1}-\bar y^*)\|^2 - (\delta^2\beta_{k}\hat{\mathcal{E}}_k  + \frac{\delta\alpha_{k}^2{\beta}_{k}^2}{2}\|B({H}^{\delta}_{k}-\bar y^*)\|^2)\nonumber\\
			\leq&({\delta}^2{\beta}_{k+1}-{\delta}^2{\beta}_{k}-{\delta}\alpha_k{\beta}_{k})(\mathcal{L}(x_{k+1},y_{k+1},\bar \lambda^*)-\mathcal{L}(x^*,y^*,\bar \lambda^*))\nonumber\\
			&+(\frac{{\delta}^2{\beta}_{k+1}{\epsilon}_{k+1}}{2}
			-\frac{{\delta}^2{\beta}_k{\epsilon}_k}{2}-\frac{{\delta}\alpha_k{\beta}_k{\epsilon}_k}{2})(\|x_{k+1}\|^2+\|y_{k+1}\|^2-\|\bar x^*\|^2-\|\bar y^*\|^2)\nonumber\\
			&-\frac{1}{2}\|Z_{k+1}^\delta-Z_{k}^\delta\|^2-\frac{1}{2}\|H_{k+1}^\delta-H_{k}^\delta\|^2 -\frac{{\delta}^2{\beta}_k{\mu}_f}{2}\|x_{k+1}-x_k\|^2-\frac{{\delta}^2{\beta}_k{\mu}_g}{2}\|y_{k+1}-y_k\|^2\nonumber\\
			&-\frac{\delta\alpha_k\beta_k\mu_g}{2}\|H_{k+1}^\delta-\bar y^*\|^2+\frac{\delta(\alpha_{k+1}^2\beta_{k+1}^2-\alpha_{k}^2\beta_k^2)}{2}\|B(H_{k+1}^\delta-\bar y^*)\|^2\nonumber\\
			&-\frac{\delta\alpha_k\beta_k\mu_f}{2}\|Z_{k+1}^\delta-\bar x^*\|^2-\frac{\delta\alpha_k\beta_k\mu_f}{2}\|x_{k+1}-Z_{k+1}^\delta\|^2-\frac{\delta\alpha_k\beta_k\mu_g}{2}\|y_{k+1}-H_{k+1}^\delta\|^2\nonumber\\
			&-(\delta\gamma-1)(\frac{\delta}{a_k}+\frac{1}{2})\|x_{k+1}-x_k\|^2-(\delta\gamma-1)(\frac{\delta}{a_k}+\frac{1}{2})\|y_{k+1}-y_k\|^2.
		\end{align*}
		From Assumption \ref{ass4}, we get
		\begin{align*}
			\delta^2\beta_{k+1}\hat{\mathcal{E}}_{k+1}&+\frac{\delta\alpha_{k+1}^2\beta_{k+1}^2}{2}\|B({H}^{\delta}_{k+1}-\bar y^*)\|^2 \leq \delta^2\beta_{k}\hat{\mathcal{E}}_k  + \frac{\delta\alpha_k^2\beta_{k}^2}{2}\|B({H}^{\delta}_{k}-\bar y^*)\|^2.
		\end{align*}
		It yields
		\begin{eqnarray*}
			\delta^2\beta_{k}\hat{\mathcal{E}}_{k} \leq \delta^2\beta_{K}\hat{\mathcal{E}}_{K} +\frac{\delta\alpha_{K}^2 \beta_{K}^2}{2}\|B({H}^{\delta}_{K}-\bar y^*)\|^2.
		\end{eqnarray*}
		This together with \eqref{energy3} implies
		\begin{eqnarray*}
			\mathcal{L}^{\epsilon_k}(x_k,y_k,{\bar\lambda^*})- \mathcal{L}^{\epsilon_k}(\bar x^*,\bar y^*,{\bar\lambda^*}) \leq
			\frac{{\beta}_{K}}{\beta_k}\hat{\mathcal{E}}_{K}+\frac{\alpha_{K}^2{\beta}_{K}^2}{2\delta\beta_k}\|B({H}^{\delta}_{K}-\bar y^*)\|^2.
		\end{eqnarray*}
		Then applying Proposition \ref{pr2}, we obtain
		\begin{align*}
			\|(x_k,y_k)- (x_{\epsilon_k},y_{\epsilon_k})\|^2 \leq&\frac{2{\beta}_{K}}{\beta_k\epsilon_k}\hat{\mathcal{E}}_{K} + \frac{\alpha_{K}^2{\beta}_{K}^2}{\delta\beta_k\epsilon_k} \|B({H}^{\delta}_{K}-\bar y^*)\|^2\nonumber\\
			&+\|\bar x^*\|^2-\|x_{\epsilon_k}\|^2+\|\bar y^*\|^2-\|y_{\epsilon_k}\|^2.
		\end{align*}
		Combining with $\lim_{k\rightarrow{+\infty}}\beta_{k}\epsilon_{k}={+\infty}$, we have
		\begin{eqnarray*}
			\lim_{k\rightarrow{+\infty}}\|(x_k,y_k)-(\bar x^*,\bar y^*)\|=0.
		\end{eqnarray*}

		The proof for Case II and Case III are similar to Theorem \ref{th4.1}, we shall omit them here for conciseness.
	\end{proof}
	
	\section{Reduction to a non-separable optimization problem}
	
	To solve the non-separable linearly constrained optimization problem (\ref{P2}), we obtain a special case of Algorithm \ref{al_1} that reduces to the following Algorithm \ref{al_4}.
	\begin{algorithm}
		\caption{Primal-dual Algorithm}
		\label{al_4}
		{\bf Initialize:} Let $x_1=x_0=\frac{1}{\gamma} Z_1, {\lambda}_1={\lambda}_0$, $\delta>0,\gamma>0$. \\
		\For{$k = 1, 2,\cdots$}{
			{\bf Step1:} Let $\theta_k=(\alpha_k+{\delta})\beta_k$ and $\eta_{f,k}=\gamma+\frac{1}{\alpha_k}+\mu_f\delta\beta_k$\\
			$\widetilde{\lambda}_k = \lambda_k -\delta\beta_k(Ax_k-b)$,\\
			$\widetilde{x}_k=x_k+\frac{Z_k-\gamma x_k}{\eta_{f,k}},$\\
			$x_{k+1} = \mathop{\arg\min}_{x\in\mathcal{X}}\left(\mathcal{L}_{\theta_k}(x,\widetilde{\lambda}_k)
			+\frac{\eta_{f,k}}{2\alpha_k\beta_k}\left\|x-\widetilde{x}_k\right\|^2
			+\frac{{\epsilon}_k}{2}\left\|x\right\|^2\right).$\\
			{\bf Step2:}${Z}_{k+1} =({\gamma+\frac{1}{\alpha_k}})x_{k+1}-\frac{1}{\alpha_k}x_k,$\\
			${Z}^\delta_{k+1}=\delta{Z}_{k+1}+(1-\delta\gamma)x_{k+1},$\\
			$\lambda_{k+1} = \lambda_k + \alpha_k\beta_{k}(A{Z}^\delta_{k+1}-b)$.
		}
	\end{algorithm}
	
	This algorithm is equivalent to discrete the dynamical system (\ref{dyn}) with only one primal variable and one dual variable. Construct the energy sequence $\{\mathcal E_k\}_{k\geq1}$ of Algorithm \ref{al_4} as
	\begin{align*}
		\mathcal{E}_k = & \delta^2\beta_k\left(\mathcal{L}(x_k,\lambda^*)-\mathcal{L}(x^*,\lambda^*)+\frac{\epsilon_k}{2}\|x_k\|^2\right) + \frac{1}{2}\|Z^\delta_k - x^*\|^2 + \frac{\delta\gamma-1}{2}\|x_k-x^*\|^2 + \frac{\delta}{2}\|\lambda_k-\lambda^*\|^2.
	\end{align*}

	And we can reformulate Assumption $\ref{ass}$ as follows.
	\begin{assumption}\label{ass6}
		Suppose that $f$ is a proper, closed and lower semi-continuous function. $\{{\beta}_k\}_{k\geq 1}$ is a positive and nondecreasing sequence and $\lim_{k\rightarrow{+\infty}}\beta_{k}={+\infty}$. $\{{\alpha}_k\}_{k\geq 1}$ is a positive sequence. For every $k\geq 1 $, the parameters $\delta,\gamma$ and the sequence ${{\beta}_k}$ satisfy
		\begin{eqnarray*}
			{\delta\gamma}-1\geq0,\qquad\delta{\beta}_{k+1}\leq\delta{\beta}_{k}+\alpha_k{\beta}_{k}.
		\end{eqnarray*}
	\end{assumption}

	Then utilizing the same argument in Theorems \ref{th3.1} and \ref{th4.1}, we obtain the following convergence rate of Algorithm $\ref{al_4}$.
	\begin{corollary}\label{th5.1}
		Suppose that Assumption \ref{ass6} holds. Let $\{(x_k,\lambda_k)\}_{k\geq 1}$ be the sequence generated by Algorithm \ref{al_4} and  $(x^*,\lambda^*)\in\Omega$.  Assume that $ \{\epsilon_{k}\}_{k\geq1} $ is a nonincreasing sequence and $\sum_{k=1}^{+\infty}\alpha_k\beta_{k}\epsilon_{k}<{+\infty}$, then we have the sequence $ \{(x_{k},\lambda_k)\}_{k\geq1} $ and the following statements:
		\begin{eqnarray*}
			\begin{aligned}
				&\mathcal{L} ( x_k,\lambda ^*)-\mathcal{L} ( x^*,\lambda ^* )=\mathcal{O}\left( \frac{1}{\beta_k} \right),\\
				\| f ( x_k)-& f ( x^*)\|=\mathcal{O} \left( \frac{1}{\beta_k} \right),\,\,\,\left \| Ax_k-b\right \|=\mathcal{O}\left ( \frac{1}{\beta_k}\right ).\\
			\end{aligned}
		\end{eqnarray*}
	\end{corollary}
	\begin{corollary}\label{th5.2}
		Suppose that Assumption \ref{ass6} holds. Let $\{(x_k,\lambda_k)\}_{k\geq 1}$ be the sequence generated by Algorithm \ref{al_4} and  $(x^*,\lambda^*)\in\Omega$.  Assume that $ \{\epsilon_{k}\}_{k\geq1} $ is a nonincreasing sequence satisfying $\sum_{k=1}^{+\infty}\alpha_k\epsilon_{k}<{+\infty}$ and $\lim_{k\rightarrow{+\infty}}\beta_{k}\epsilon_{k}={+\infty}$, then
		\begin{eqnarray*}\label{q5.1}
			\underset{k\rightarrow{+\infty}}{\lim\inf}\|x_k- \bar x^*\|=0.
		\end{eqnarray*}
		Further, if there exists an integer $K\geq1$ such that the sequence $\{x_k\}_{k\geq K}$ stays in either the open ball ${B}\left( 0,\| \bar{x}^* \|\right)$ or its complement, then
		\begin{eqnarray*}\label{q5.2}
			\lim_{k\rightarrow{+\infty}}\|x_k-\bar x^*\|=0.
		\end{eqnarray*}
	\end{corollary}
	\begin{remark}
		For $\beta_k > k^2$, Algorithm \ref{al_4} achieves the convergence rate $O(1/k^{p})$ with $p > 2$, which faster than $O(1/k^2)$ result in \cite{42}. Based on Assumption \ref{ass6}, we can establish an exponential convergence rate of $O( (1 + \frac{\alpha_{\min}}{\delta})^{-k} )$ for the primal-dual gap, where $\alpha_k \geq \alpha_{\min} > 0$. When $\{\alpha_k\}_{k\geq1}$ is fixed to one, this rate coincides with the $O((1+\frac{1}{\delta})^{-k})$ rate presented in \cite{13}.
	\end{remark}
	
	\section{Numerical experiments}
	
	To validate the effectiveness of Algorithm \ref{al_3} (referred to as PDSA), we conduct two numerical experiments in this section. All compared algorithms are parameterized according to their theoretical convergence guarantees. The optimal values are computed via CVX. All codes are performed on a PC (with 3.10GHz Intel Core i5-11300H and 16GB memory).
	\begin{example} \cite[Example 1]{txz} Consider the least absolute deviation (LAD) regression problem:
		\begin{equation}\label{szsy1}
			\underset{x\in \mathcal{R}^n}\min\,f(x)+\|Mx-b\|_1,
		\end{equation}
		where $M \in \mathbb{R}^{m \times n}$ and $b \in \mathbb{R}^m$, with the dimension constraint $m \ll n$. We examine two cases, each with two dimensions. In Case 1, the matrix $M$ is generated from $\mathcal{N}(0,1)$ and its rows are normalized to unit norm. In Case 2, the matrix $M$ is generated identically to Case 1, but with $50\%$ correlated columns in $K$. The vector $b$ is constructed as $b = M\underline{x} + \varrho$, where $\underline{x}$ denotes a sparse vector and $\varrho$ represents Gaussian noise with variance $\sigma^2 = 1e-4$ and $10\%$ nonzero entries. For ADMM, a reformulation is applied with $y:= Mx-b$ in (\ref{szsy1}). In contrast, for algorithms such as PDSA, Semi-APD, New-PDA1 and New-PDA2, we reformulate problem (\ref{szsy1}) into the constrained form:
		\begin{equation*}
			\begin{array}{ll}
				&\min\limits_{x, y}~f(y)+\|x-b\|_1\\
				&~~\mbox{s.t.}~~~x-My=0.
			\end{array}
		\end{equation*}

		In the following, we use the composite objective residual, the violation of feasibility and the objective residual to assess convergence behaviors.

		$\mathbf{ Case~ I}(LAD-LASSO)$: We define the function $f(x)$ in equation \eqref{szsy1} as $ f(x) = \lambda \| x \|_1 $ with a regularization parameter $\lambda:=0.2$. We conduct a comparison of these algorithms below:
		\begin{itemize}
			\item the alternating direction method of multipliers (ADMM) \cite{11}: $\rho=0.1$ or $1$ or $10$;
			\item the Chambolle-Pock's method (CP) \cite{cxz}: $\rho=0.1$ or $1$ or $10$, $\gamma=0.999$, $\theta=1$, $\tau=\frac{\gamma}{\rho\|M\|^2}$;
			\item the new primal-dual algorithm (New-PDA1) \cite[Algorithm 1]{txz}: $c=1$ or $2$, $\gamma=0.999$, $\rho_0=\frac{ \sqrt{\gamma/(1 - \gamma)}\|y_0 - y^*\|}{\|M\|\|x_0 - x^*\|}$, $\rho_k=\frac{\rho_0}{\tau_k}$, $\beta_k=\frac{\gamma}{\|M\|^2\rho_k}$, $\eta_k=(1-\gamma)\rho_k$, $\tau_k=\frac{c}{k+c}$;
			\item the semi-implicit scheme (Semi-APD) \cite[scheme (3.20)]{30}: $\alpha_k=\frac{\sqrt{\theta_k\beta_k}}{\|M\|}$, $\theta_k=\frac{\theta_{k-1}}{1+\alpha_{k-1}}$, $\gamma_k=\frac{\gamma_{k-1}}{1+\alpha_{k-1}}$, $\beta_k=\frac{\beta_{k-1}}{1+\alpha_{k-1}}$;
			\item the primal-dual splitting algorithm (PDSA) [Algorithm \ref{al_3}]: $\gamma=2$, $\delta=0.6$, $\alpha_k=\frac{1}{k}$, $\beta_k=k$, $\epsilon_k=\frac{1}{k^3}$.
		\end{itemize}

		We test different parameter settings for the ADMM and CP algorithms. Under these settings, ADMM and CP exhibit an ergodic convergence rate of $O(1/k)$, while New-PDA1, Semi-APD and our PDSA achieve a non-ergodic convergence rate of $O(1/k)$. In this numerical experiment, algorithms with non-ergodic convergence are faster than those with ergodic convergence under the same convergence rate. As shown in Figure 1, our PDSA achieves superior performance compared to other algorithms.
		\begin{figure}
			\centering
			\includegraphics[width=.9\textwidth]{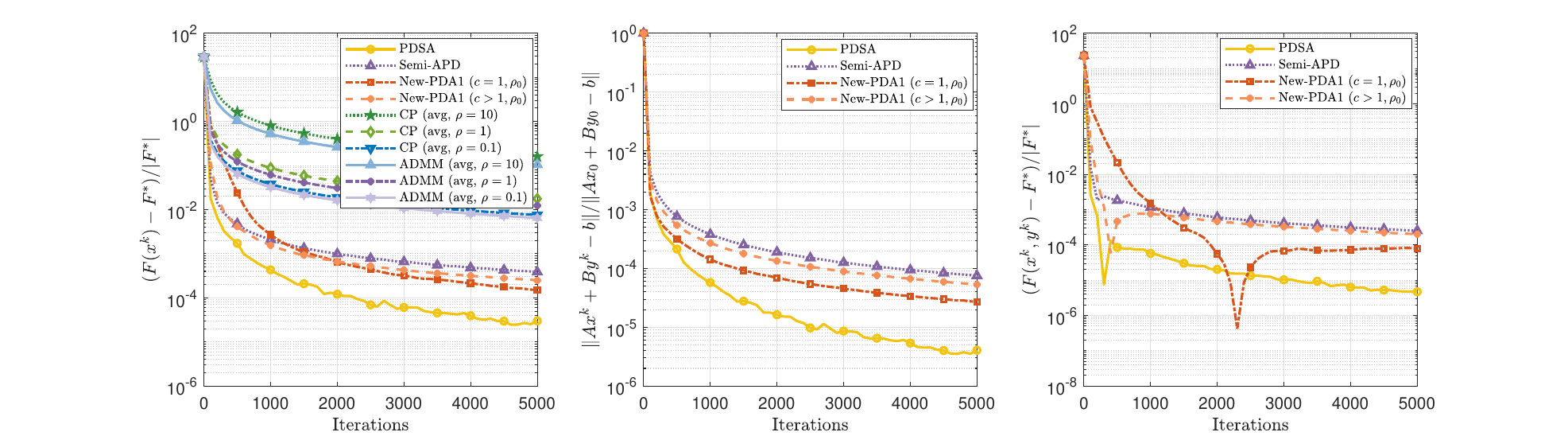}
			\caption{The convergence behaviors of algorithms for problem (\ref{szsy1}) when $f$ is a convex function. The problem size is $(m, n) = (300, 3000)$.}
			\label{fig1}
		\end{figure}
		
		$\mathbf{ Case~ II}(LAD-Elastic Net)$: We define the function $f(x)$ in equation \eqref{szsy1} as $ f(x) = \lambda \| x \|_1+\frac{\mu}{2}\| x \|^2 $ with $\lambda:=0.2$ and $\mu:= 0.2$. We compare these algorithms:
		\begin{itemize}
			\item the strongly convex variant of CP (CP-scvx) \cite{cxz2}: $\rho=\rho_{cp}=\frac{0.1}{\|M\|}$ or $\frac{1}{\|M\|}$ or $\frac{10}{\|M\|}$, $\gamma=0.999$, $\theta=1$, $\tau=\frac{\gamma}{\rho\|M\|^2}$;
			\item the new primal-dual algorithm (New-PDA2) \cite[Algorithm 2]{txz}: $(i)$ $c=2$, $\gamma=0.999$, $\Gamma=2-\frac{1}{\gamma}$, $\rho_0=\rho_0^1=\frac{ \Gamma\mu}{2\|M\|^2}$, $\rho_k=\frac{\rho_0}{\tau_k^2}$, $\beta_k=\frac{\Gamma}{\rho_k\|M\|^2}$, $\eta_k=(1-\gamma)\rho_k$, $\tau_0=1$, $\tau_{k+1}=\frac{\tau_{k}}{2}(\sqrt{\tau_k^2+4}-\tau_k)$; $(ii)$ $c=4$, $\gamma=0.75$, $\Gamma=2-\frac{1}{\gamma}$, $\rho_0=\rho_0^2=\frac{ c(c-1)\Gamma\mu}{(2c-1)\|M\|^2}$, $\rho_k=\frac{\rho_0}{\tau_k^2}$, $\beta_k=\frac{\Gamma}{\rho_k\|M\|^2}$, $\eta_k=(1-\gamma)\rho_k$, $\tau_{k+1}=\frac{c}{k+1+c}$;
			\item the semi-implicit scheme (Semi-APD) \cite[scheme (3.20)]{30}: $\alpha_k=\frac{\sqrt{\theta_k\beta_k}}{\|M\|}$, $\theta_k=\frac{\theta_{k-1}}{1+\alpha_{k-1}}$, $\gamma_k=\frac{\gamma_{k-1}+\mu\alpha_{k-1}}{1+\alpha_{k-1}}$, $\beta_k=\frac{\beta_{k-1}+\mu\alpha_{k-1}}{1+\alpha_{k-1}}$;
			\item the primal-dual splitting algorithm (PDSA) [Algorithm \ref{al_3}]: $\gamma=3.4$, $\delta=0.3$, $\alpha_k=\frac{1}{k}$, $\beta_k=\frac{\mu k^2}{3\|M\|^2}$, $\epsilon_k=\frac{1}{\alpha_k\beta_kk^3}$.
		\end{itemize}

		All competing algorithms can achieve an $O(1/k^2)$ convergence rate where $f$ is strongly convex while $g$ remains convex. Figure \ref{fig2} demonstrates that our algorithm attains a lower composite objective residual, violation of feasibility and objective residual with respect to the number of iterations, indicating a faster convergence rate relative to the competing methods in this case.
		\begin{figure}
			\centering
			\includegraphics[width=.9\textwidth]{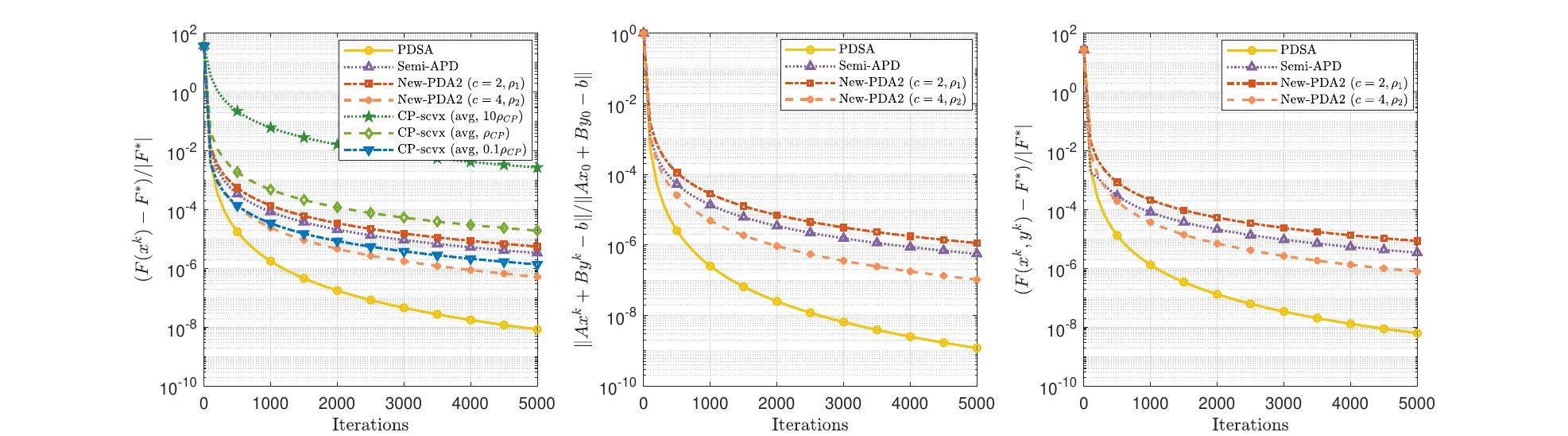}
			\caption{The convergence behaviors of algorithms for problem (\ref{szsy1}) when $f$ is a strongly convex function. The problem size is $(m, n) = (300, 3000)$.}
			\label{fig2}
		\end{figure}

		Furthermore, a different dimensional setting with $m=400$ and $n=5000$ is considered in Figure \ref{fig3}. Clearly, PDSA continues to outperform the other algorithms.
		\begin{figure}
			\centering
			\includegraphics[width=.9\textwidth]{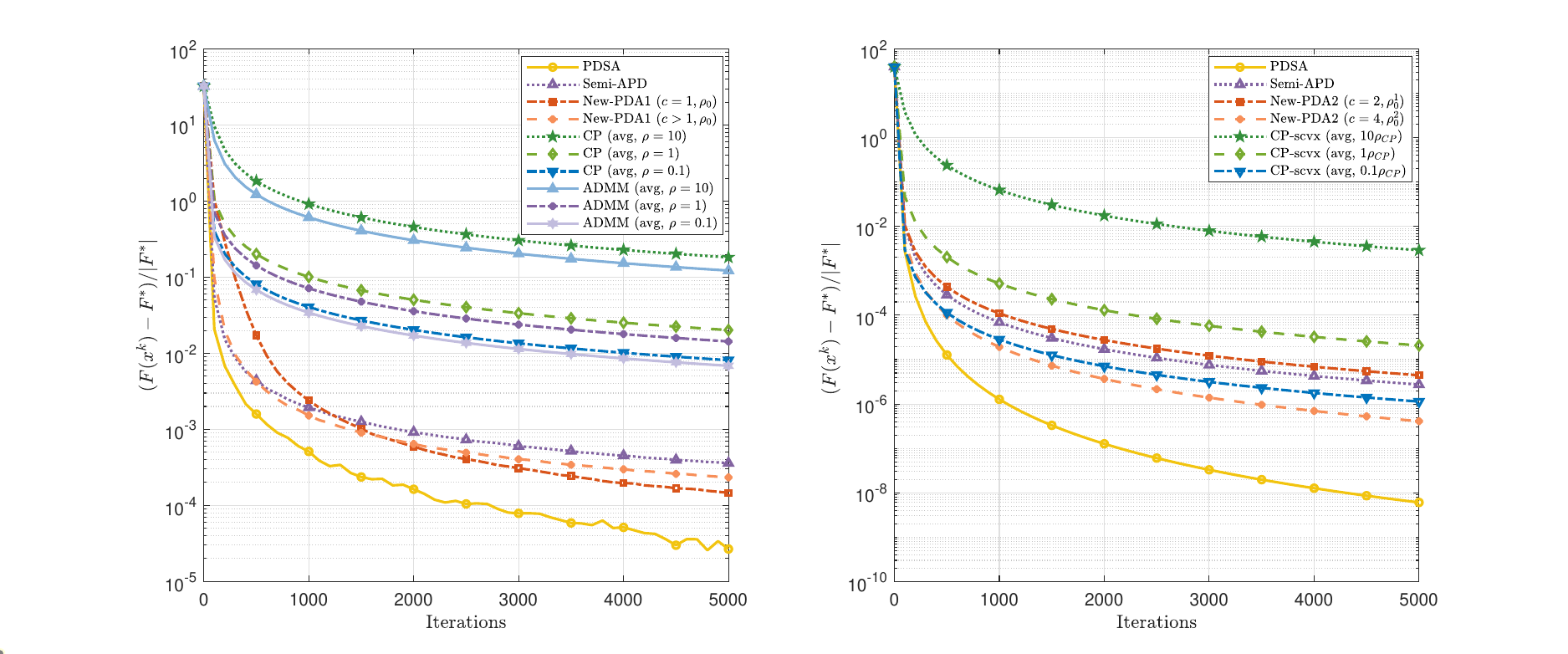}
			\caption{The convergence behaviors of algorithms for problem (\ref{szsy1}). Left: Case I (general convex); Right: Case II (partially strongly convex). The problem size is $(m, n) = (400, 5000)$.}
			\label{fig3}
		\end{figure}
		
	\end{example}
	\begin{example} Let $x:=(x_1,x_2,x_3)\in \mathbb{R}^3$ and $y:=(y_1,y_2,y_3)\in \mathbb{R}^3$.
		Consider a particular $l_1-l_1$ minimization problem:
		\begin{equation*}
			\underset{x\in \mathcal{R}^n}\min\,\lambda\|x\|_1+\|Mx-b\|_1,
		\end{equation*}
		where $M = \text{diag}(p, q, r)$ is a diagonal matrix with $p,q,r\in \mathbb{R}\backslash\{0\}$ and $b\in\mathbb{R}^3$ is a constant vector with all entries equal to $d$. This problem is equivalent to:
		\begin{equation*}
			\begin{array}{ll}
				&\min\limits_{x, y}~\varPhi  (x,y):=\lambda\|y\|_1+\|x-b\|_1\\
				&~~~\mbox{s.t.}~~~~~x-My=0.
			\end{array}
		\end{equation*}
		For different choices of $p, q ,r ,\lambda$ and $d$, we observe that the sequence $\{(x_k,y_k)\}_{k\geq1}$ generated by PDSA always converges to the minimum norm solution. In the following numerical experiments, we take the initial point $y^{(0)}=(-0.5, 0.5, 1)^T$, $x^{(0)}=My^{(0)}$ and $\lambda^{(0)}=(0, 0, 0)^T$. We set the parameters for two cases.
		
		$\mathbf{ Case~ I}$: Let $\lambda=3$, $d=2$, $p=2$, $q=3$, $r=1$ and $\gamma=2$, $\delta=0.7$, $\alpha_k=\frac{1}{k}$, $\beta_k = k$, $ \epsilon_k = \frac{1} {\sqrt{k}}$ in PDSA. In this case, the optimal solution set of the problem is $S = \{(0, 3y_2, 0, 0, y_2, 0) : y_2 \in [0, \frac{2}{3}]\}$ and the optimal value is $6$. Moreover, the minimal norm solution of the problem is $(\bar{x}^*, \bar{y}^*) = (0,0,0,0,0,0)$.
		
		$\mathbf{ Case~ II}$: Let $\lambda=2$, $d=2$, $p=1$, $q=1$, $r=2$ and $\gamma=2$, $\delta=0.7$, $\alpha_k=\frac{1}{k}$, $\beta_k = k$, $ \epsilon_k = \frac{1} {\sqrt{k}}$ in PDSA. In this case, the optimal solution set of the problem is $S = \{(0, 0, 2y_3, 0, 0, y_3) : y_3 \in [0, 1]\}$ and the optimal value is $6$. Moreover, the minimal norm solution is $(\bar{x}^*, \bar{y}^*) = (0,0,0,0,0,0)$.
		\begin{figure}
			\centering
			\includegraphics[width=.9\textwidth]{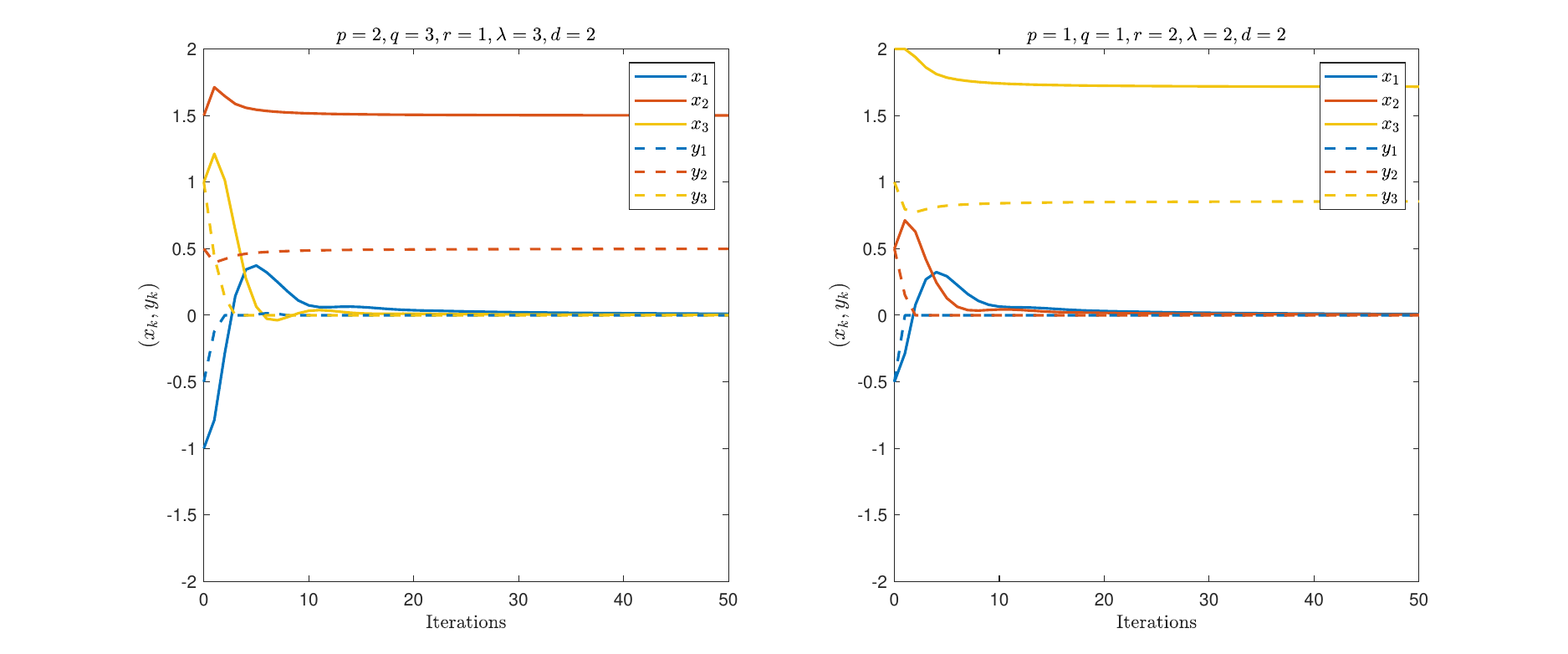}
			\caption{Numerical results of PDSA under the different choices of parameters with $\epsilon_k=0$.}
			\label{bfigzx}
		\end{figure}
		
		\begin{figure}
			\centering
			\includegraphics[width=.9\textwidth]{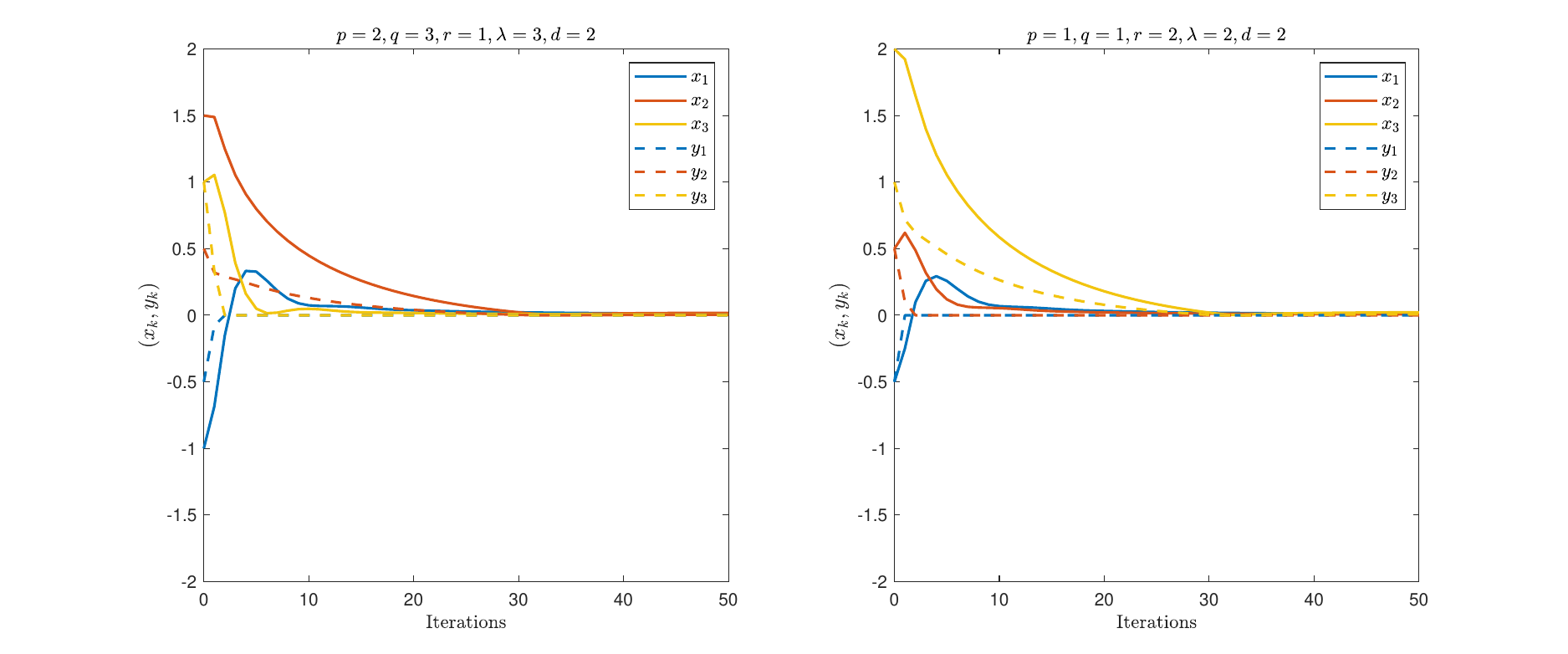}
			\caption{Numerical results of PDSA under the different choices of parameters with $\epsilon_k\neq 0$.}
			\label{figzxfsj}
		\end{figure}
		
		Under these two parameter settings considered in Cases I and II, the numerical results shown in Figures \ref{figzxfsj} and \ref{bfigzx} demonstrate that the sequences generated by PDSA with Tikhonov regularization converge to its minimum norm solution.
	\end{example}
	
	\section{Conclusion}
	
	In this paper, we propose novel numerical algorithms for solving separable convex optimization problems by discretizing the mixed-order dynamical system $(\ref{dyn})$, which incorporates time scales and a Tikhonov regularization. Our analysis not only establishes rapid convergence rates for the objective function, the primal-dual gap and the feasibility violation, but also demonstrates strong convergence of these sequences generated by proposed algorithms under the nonsmooth assumptions on $f$ and $g$.

	As part of future work, several promising directions deserve further exploration. First, given that classical splitting schemes cannot be directly extended to multi-block settings \cite{12}, an interesting question is how to discretize an appropriate dynamical system to develop novel splitting algorithms for multi-block convex optimization. Furthermore, inspired by \cite{24}, a natural extension would be to investigate whether incorporating a Hessian-driven damping term within such an algorithmic framework can reduce oscillations and achieve comparable or even better convergence rates.
	
	%%%%%%%%%%%%%%%%%%%%%%%%%%%%%%%%%%%%%%%%%%%%%%%%%%%%%%%%%
	\appendix
	
	\section*{CRediT authorship contribution statement}
	{\bf Geng-Hua Li:} Conceptualization, Supervision,  Funding acquisition, Writing.
	{\bf Hai-yi Zhao:} Conceptualization, Software, Visualization, Writing.
	{\bf Xiangkai Sun:} Methodology, Supervision,
	Funding acquisition, Writing.
	
	\section*{Declaration of competing interest}
	
	\small{ The authors declare that they have no known competing financial interests or personal relationships that could have appeared to influence the work reported in this paper.}
	
	\section*{Data availability}
	
	\small{ The authors confirm that all data generated or analysed during this study are included in this article.}

\end{document}